\documentclass{amsart}
\usepackage{amsmath,amssymb,amscd,latexsym}

\usepackage[usenames, dvipsnames]{color}

\usepackage[all]{xy}
\newcommand{\F}{{\mathbb{F}}}

\newcommand{\Sph}{{\mathbb{S}}}
\newcommand{\Z}{{\mathbb{Z}}}

\newcommand{\op}{\mathrm{op}}
\newcommand{\colim}{\operatorname*{colim}}
\newcommand{\hocolim}{\operatorname*{hocolim}}
\newcommand{\holim}{\operatorname*{holim}}

\newcommand{\compl}{\hat{(\cdot)}}

\newcommand{\Gal}{\mathrm{Gal}}

\newcommand{\Hom}{\mathrm{Hom}}

\newcommand{\map}{\mathrm{map}}

\newcommand{\Map}{\mathrm{Map}}

\newcommand{\Mapg}{\mathrm{Map}_G}

\newcommand{\hEh}{\hat{\mathcal E}}

\newcommand{\Enik}{E'_{n,I_k}}

\newcommand{\hHhpg}{\hat{{\mathcal H}}_{\ast G}}

\newcommand{\Sh}{{\mathcal S}}
\newcommand{\Shp}{{\mathcal S}_{\ast}}
\newcommand{\Shg}{{\mathcal S}_G}
\newcommand{\Shpg}{{\mathcal S}_{\ast G}}

\newcommand{\hSh}{\hat{\mathcal S}}
\newcommand{\hShg}{\hat{\mathcal S}_G}

\newcommand{\hShp}{\hat{\mathcal S}_{\ast}}
\newcommand{\hShpg}{\hat{\mathcal S}_{\ast G}}

\newcommand{\hSp}{\mathrm{Sp}(\hShp)}
\newcommand{\hSpg}{\mathrm{Sp}(\hShpg)}

\newcommand{\Sp}{\mathrm{Sp}(\Sh_{\ast})}
\newcommand{\Spg}{\mathrm{Sp}(\Shpg)}

\newtheorem{theorem}{Theorem}[section]
\newtheorem{lemma}[theorem]{Lemma}

\newtheorem{defn}[theorem]{Definition}

\newtheorem{cor}[theorem]{Corollary}
\theoremstyle{definition}

\newtheorem{remark}[theorem]{Remark}

\begin{document}

\title{Profinite and discrete $G$-spectra and \\iterated homotopy fixed points}

\author{Daniel G. Davis}
\thanks{The first author was partially supported by a grant (LEQSF(2008-11)-RD-A-27) from the Louisiana Board of Regents during the first half of his work on this 
paper.}
\address{Department of Mathematics, University of Louisiana at Lafayette, Lafayette, LA 70504, USA}
\email{dgdavis@louisiana.edu}

\author{Gereon Quick}\thanks{The second author was partially supported by Research Fellowship QU 317/1 of the German Research Foundation (DFG)}
\address{Department of Mathematical Sciences, NTNU, NO-7491 
Trondheim, Norway}
\email{gereon.quick@math.ntnu.no}

\date{}
\begin{abstract}
For a profinite group $G$, let $(\text{-})^{hG}$, $(\text{-})^{h_dG}$, and $(\text{-})^{h'G}$ 
denote continuous homotopy fixed points for 
profinite $G$-spectra, discrete $G$-spectra, and continuous $G$-spectra 
(coming from towers of discrete $G$-spectra), respectively. We establish some connections between the first two notions, and by using Postnikov towers, for $K \vartriangleleft_c G$ (a closed normal subgroup), give various conditions for when the iterated homotopy fixed points $(X^{hK})^{hG/K}$ exist and are $X^{hG}$. For 
the Lubin-Tate spectrum $E_n$ and $G <_c G_n$, the extended 
Morava stabilizer group, our results 
show that $E_n^{hK}$ is a profinite $G/K$-spectrum with 
$(E_n^{hK})^{hG/K} \simeq E_n^{hG}$, by an argument that 
possesses a certain technical simplicity not enjoyed by either the 
proof that $(E_n^{h'K})^{h'G/K} \simeq E_n^{h'G}$ or 
the Devinatz-Hopkins proof (which requires $|G/K| < \infty$) of 
$(E_n^{dhK})^{h_dG/K} \simeq E_n^{dhG}$, where $E_n^{dhK}$ is a 
construction that behaves like continuous homotopy fixed points. Also, we prove that (in general) the 
$G/K$-homotopy fixed point spectral sequence for $\pi_\ast((E_n^{hK})^{hG/K})$, 
with $E_2^{s,t} = H^s_c(G/K; \pi_t(E_n^{hK}))$ 
(continuous cohomology), is isomorphic to both 
the strongly convergent 
Lyndon-Hochschild-Serre spectral sequence of Devinatz for $\pi_\ast(E_n^{dhG})$
%, with $E_2^{s,t} = H^s_c(G/K; \pi_t(E_n^{dhK}))$, 
and the descent spectral sequence for 
$\pi_\ast((E_n^{h'K})^{h'G/K})$. 
%, with $E_2^{s,t} = H^s_c(G/K; \pi_t(E_n^{h'K}))$.  
\end{abstract}
\maketitle

\pagestyle{myheadings}
\markboth{DANIEL G. DAVIS AND GEREON QUICK}{PROFINITE 
AND DISCRETE $G$-SPECTRA \& 
ITERATED HOMOTOPY FIXED POINTS}

\section{Introduction}
If $G$ is a (discrete) group acting on a spectrum $X$, one can form the homotopy fixed point spectrum $X^{hG}$. The spectrum $X^{hG}$ is defined as the $G$-fixed points of the function spectrum $F(EG_+,X)$, where $EG$ is a free contractible $G$-space. If $G$ carries a (non-discrete) topology with respect to which the action on $X$ is in some sense continuous, one would like to 
have constructions of (i) a continuous homotopy fixed point spectrum 
that respects the 
continuous action, and (ii) an associated homotopy fixed point spectral sequence whose $E_2$-term consists of continuous cohomology groups. 
When $G$ is a profinite group, by building on earlier work -- by Jardine \cite{jardine2, jardine} and Thomason \cite{thomason} (see the helpful paper \cite{mitchell} 
by Mitchell), and Goerss \cite{goerss}, in the 
case of discrete objects, and by the second author \cite{gspaces, gspectra} in the case 
of profinite objects, this construction problem has been studied 
(a) for discrete and continuous $G$-spectra in \cite{davis, behrensdavis}, 
and (b) for profinite $G$-spectra in \cite{hfplt}. Motivated by the fact that a 
profinite $G$-set that is finite is also a discrete $G$-set, one of the purposes of this paper is to compare approaches 
(a) and (b) in certain cases 
and let the tools of one approach supplement the 
techniques of the other.

\subsection{Iterated continuous homotopy fixed points and profinite $G$-spectra}\label{oneone}
It is a standard fact that if $H$ is any (discrete) group and $N$ is any normal subgroup, 
then for any $H$-space $\mathcal{X}$, the space $\mathcal{X}^{hN}$ can be identified 
with the $H/N$-space $\mathrm{Map}^N\mspace{-2.8mu}(EH, \mathcal{X})$, so that the iterated homotopy fixed point space $(\mathcal{X}^{hN})^{hH/N}$ is defined. Furthermore, it is well-known that $(\mathcal{X}^{hN})^{hH/N}$ is just $\mathcal{X}^{hH}$. 

We continue to let $G$ be a profinite group and we consider 
the setting of profinite $G$-spectra. If $K$ is a closed normal subgroup of 
% an arbitrary profinite group $G$ 
$G$ and $X$ is 
any profinite $G$-spectrum, the fixed points satisfy the equality 
$X^G=(X^K)^{G/K}$. 
This identity and the aforementioned fact 
that $(\mathcal{X}^{hN})^{hH/N} \simeq \mathcal{X}^{hH}$ motivate the following question about continuous homotopy fixed points $(-)^{hG}$ for profinite 
$G$-spectra: 
is there an equivalence
\begin{equation}\label{iteratedeq}
X^{hG} \simeq (X^{hK})^{hG/K}
\end{equation} 
between these two spectra? (In (\ref{iteratedeq}) and henceforth, whenever the 
group $P$ 
is profinite, the notation $(-)^{hP}$ denotes homotopy fixed points for profinite 
$P$-spectra.)
%Thus, another purpose of this paper 
%is to take this familiar situation as a model and, for a profinite group $G$, 
%consider
%the formation of iterated continuous homotopy fixed points in the setting of profinite $G$-spectra.

%Let us now go a step further and consider iterated homotopy fixed points.  

In the setting of profinite groups and 
for any object in some category of $G$-spectra, the above question 
% represented by (\ref{iteratedeq}) 
was first asked in \cite[page 130]{devinatz2} and, for the category of discrete 
$G$-spectra, the question was studied in detail in \cite{davis2, behrensdavis, delta}. The equivalence in (\ref{iteratedeq}) 
would simplify the analysis of the homotopy fixed points under $G$ by reducing it to the study of those under proper closed normal subgroups $K$ and the quotients $G/K$. 

Unfortunately, it is in general not known that the homotopy fixed point spectrum $X^{hK}$ has the same topological characterization as the profinite $G$-spectrum $X$. For example, when 
the profinite group $G/K$ is not finite, $X^{hK}$ is in general not known to be a profinite $G/K$-spectrum. These basic issues in the problem of iteration are considered in more 
detail in this paper in Sections \ref{fourone} and \ref{fourtwo}, and 
to make progress on them, the following terminology is helpful. 

\begin{defn}\label{vcd}
The profinite group $G$ has finite cohomological dimension if there exists a positive 
integer $r$ such that the continuous cohomology $H^s_c(G;M) = 0$, for all $s \geq r$ and every discrete 
$G$-module $M$ $\mathrm{(}$this notion is more commonly referred to as finite strict cohomological dimension$\mathrm{)}$. We say that $G$ has finite virtual cohomological dimension 
if $G$ contains an open subgroup $U$ that has finite cohomological dimension. 
\end{defn}

To address the situation described just before Definition \ref{vcd}, we provide various sets of sufficient conditions on 
$G$ and $X$ that allow for the formation of the iterated continuous homotopy fixed points $(X^{hK})^{hG/K}$ and 
the obtaining of equivalence \eqref{iteratedeq}. 
A useful tool for this work is a comparison result between profinite and discrete homotopy fixed points: for certain $X$ that are built out of simplicial finite discrete 
$G$-sets (see Definition \ref{fgspectrum}), the two notions of continuous homotopy fixed points are each defined, 
and under the assumption that $G$ has finite virtual cohomological dimension, 
Theorem \ref{profdisc} gives an equivalence 
\[X^{hG} \simeq X^{h_dG},\] where on the right side, $X$ is regarded as a discrete $G$-spectrum and $(-)^{h_dG}$ denotes homotopy fixed points for 
discrete $G$-spectra. 

%Another useful tool in our work is 
%the notion of hyperfibrant discrete $G$-spectrum from 
%\cite{davis2}; we recall the definition and main properties of this concept in 
%Remark \ref{hyperfibrant}. 
Another very helpful tool is the notion of a 
$K$-Postnikov $G$-spectrum, for $K$ a closed normal 
subgroup of $G$: such an object is 
a certain type of profinite 
$G$-spectrum that has well-behaved Postnikov sections with respect to $K$ 
(see Definition \ref{postnikovK} for the details). For these 
objects, we have the following iteration result. 

\begin{theorem}[{Theorem \ref{fiterated}}]\label{firstmain}
If $G/K$ has finite virtual cohomological dimension and $X$ is a 
$K$-Postnikov $G$-spectrum, then there is an equivalence 
\[X^{hG} \simeq (X^{hK})^{hG/K}.\]
\end{theorem} 

For the duration of this paragraph, we assume that $G$ is strongly complete (equivalently, if every subgroup of $G$ that has finite 
index is open in $G$) and we let $Z$ be a (naive) $G$-spectrum with stable homotopy groups $\pi_\ast(Z)$ degreewise finite. 
These two conditions imply that 
$Z$ can be realized by the profinite $G$-spectrum $F^s_GZ$ 
(see Theorem \ref{Gstablefinitecompletion}, which is recalled from \cite{gspectra}, for the details). In Theorem 
\ref{pifiniteiterated}, we use Theorem \ref{firstmain} to show that if 
$G/K$ has finite virtual cohomological dimension and $F^s_GZ$ 
is a $K$-Postnikov $G$-spectrum, then the equivalence in 
(\ref{iteratedeq}) is valid for $F^s_GZ$ and, as carefully explained (as the theory 
gradually develops) in Section \ref{four}, it is natural to write 
\[(Z^{hK})^{hG/K} \simeq Z^{hG}.\] 

By Lemma \ref{groupresult}, whenever a profinite group $G$ is a closed subgroup of a $p$-adic analytic profinite group, then, as needed in the above paragraph, 
$G$ (and $G/K$) is strongly complete and 
$G/K$ has finite virtual cohomological dimension. In Theorem \ref{postnikovresult}, Corollary \ref{goodtool}, and Remark \ref{symondsremark}, we give various conditions on a profinite $G$-spectrum that imply it is a $K$-Postnikov $G$-spectrum. To illustrate the results mentioned 
in this paragraph and the preceding one, we give 
the following interesting result.

\begin{theorem}\label{notinlatersection}
Let $p$ be any fixed prime. If the profinite group $G$ is $p$-adic analytic, with 
$K$ a closed normal subgroup, 
and $Z$ is any $G$-spectrum with 
$\pi_t(Z)$ a finite $p$-torsion abelian group for every $t \in \mathbb{Z}$, then 
\[(Z^{hK})^{hG/K} \simeq Z^{hG}.\]
\end{theorem}

\begin{remark}
Theorem \ref{notinlatersection} is easily stated and 
compelling, but its brief proof involves a little page-turning: 
by \cite[Theorem 9.6]{dsms}, $K$ is $p$-adic analytic, 
so it is of type $p$-$\mathbf{FP}_\infty$ (see \cite[page 377; Theorem 5.1.2]{symondsweigel}); by Theorem \ref{Gstablefinitecompletion}, for every $t$, there is an 
isomorphism $\pi_t(Z) \cong \pi_t(F^s_GZ)$ of $G$-modules and $\pi_t(F^s_GZ)$ is a discrete 
$G$-module (see just after Definition \ref{fgspectrum}); by Remark \ref{symondsremark}, 
$H^s_c(K; \pi_t(F^s_GZ))$ is finite for all $s$ and $t$; by Corollary \ref{goodtool}, $F^s_GZ$ is a $K$-Postnikov $G$-spectrum.
\end{remark}

Let $J$ be any small 
category. Our next step is to apply the above tools 
and results to the homotopy limit of $\{X_\beta\}_{\beta \in J}$, 
a diagram of 
$K$-Postnikov $G$-spectra that is natural in $\beta$. This concept 
(with matching name) is formalized in Definition \ref{KGbeta}; as explained in 
Remark \ref{aboutprevious}, the definition 
just unpacks the meaning present in the name in a natural way. 
For such diagrams, we obtain the following result.

\begin{theorem}[{Theorem \ref{iteratedthm}}]\label{nextmain}
If $G$ is a profinite group with $G/K$ having finite virtual cohomological 
dimension, then there is an equivalence 
\[\Bigl(\mspace{-2mu}\bigl(\holim_\beta X_\beta\bigr)^{\negthinspace hK}\Bigr)^{\negthinspace hG/K} \simeq 
\bigl(\holim_\beta X_\beta\bigr)^{\negthinspace hG}\] and a conditionally convergent 
spectral sequence
\[E_2^{s,t} = H^s_c\bigl(G/K; \pi_t\bigl((\holim_\beta X_\beta)^{hK}\bigr)\bigr) 
\Rightarrow \pi_{t-s}\bigl((\holim_\beta X_\beta)^{hG}\bigr).\]
\end{theorem} 

In the above spectral sequence, the $E_2$-term is continuous cohomology 
with coefficients the profinite $G/K$-module 
$\pi_t\bigl((\holim_\beta X_\beta)^{hK}\bigr)$. 

In 
Corollary \ref{iteratedpifinitecor}, we use Theorem \ref{nextmain} to give 
a result about when $\holim_{\beta \in J} Z_\beta$ (as above, $J$ is a small 
category), where $\{Z_\beta\}_{\beta \in J}$ is a diagram of $G$-spectra that 
have all homotopy groups finite (and each $Z_\beta$ must be a fibrant spectrum), 
can be realized by a profinite $G$-spectrum for which the equivalence in 
(\ref{iteratedeq}) holds.

\begin{remark}
If $K$ is nontrivial and non-open in a profinite group $G$ and $Y$ is a discrete $G$-spectrum that is not coconnective (that is, there is no $k$ such that $\pi_t(Y) = 0$ whenever 
$t > k$), the 
most fruitful condition for obtaining the equivalence 
\[(Y^{h_dK})^{h_dG/K} \simeq Y^{h_dG}\] 
is the requirement that $Y$ be a hyperfibrant discrete $G$-spectrum (see \cite{davis2}; some details about this are recalled in Remark \ref{hyperfibrant}). But in general, hyperfibrancy is a highly nontrivial condition, and 
in practice, one of the properties that has been used to show that it holds is that 
of being a certain type of Galois extension (for example, the conclusion of the first sentence in \cite[proof of Lemma 6.3.6]{behrensdavis} 
gives the first isomorphism in \cite[proof of Lemma 7.1]{davis2}, and 
this last lemma is needed to prove the hyperfibrancy result 
\cite[Corollary 7.2]{davis2}). Our work in this paper shows that 
the Postnikov tower of a profinite $G$-spectrum is quite useful in the 
study of iterated continuous homotopy fixed points, 
and it has allowed us to avoid dealing with hyperfibrancy for non-coconnective 
spectra and Galois extensions.
\end{remark}

\subsection{An application of our theory: the action of $G_n$ on $E_n$}
For the study of continuous actions by profinite groups in homotopy theory, a fundamental and motivating example 
is the action of the extended Morava stabilizer group $G_n$ on the Lubin-Tate spectrum $E_n$. We quickly review this example.

Let $p$ be a fixed prime, $n\geq 1$ an integer and $\F_{p^n}$ the field with $p^n$ elements. Let $S_n$ be the $n$th Morava stabilizer group, i.e. the automorphism group of the height $n$ Honda formal group law over $\F_{p^n}$. We denote by $\Gal(\F_{p^n}/\F_p)$ the Galois group of $\F_{p^n}$ over $\F_p$ and let \[G_n=S_n \rtimes \Gal(\F_{p^n}/\F_p)\] 
be the semi-direct product. Let $K(n)$ be the $n$th Morava $K$-theory spectrum with $K(n)_*=\F_p[v_n^{\pm1}]$, with $|v_n|=2(p^n-1)$. The Lubin-Tate spectrum $E_n$ is the $K(n)_*$-local Landweber exact spectrum whose coefficients are given by \[E_{n\ast}=W(\F_{p^n})[[u_1,\ldots,u_{n-1}]] [u,u^{-1}],\] where $W(\F_{p^n})$ is the ring of Witt vectors of the field $\F_{p^n}$, $|u_i|=0$ for all $i$ and $|u|=-2$. The group $G_n$ acts 
on the graded ring $E_{n\ast}$, and by Brown representability, this action is induced by an action of $G_n$ on $E_n$ by maps of ring spectra in the stable homotopy category. Furthermore, by work of Goerss, Hopkins and Miller 
(see \cite{goersshopkins, rezk}), this homotopy action is induced by 
an action of $G_n$ on $E_n$ 
before passage to the stable homotopy category. 

Now $G_n$ is a profinite group and each homotopy group $\pi_t E_n$ has the structure of a continuous profinite $G_n$-module. From Morava's change of rings theorem we know that the continuity of the action of $G_n$ on each $\pi_tE_n$ is an important property for stable homotopy theory (to view Morava's theorem in action, see, for example, \cite[Section 2]{devinatzhopkins}). The most succinct way to convey the importance of this continuous 
action is to note that for any finite spectrum $Y$, there is a strongly convergent homotopy fixed point spectral 
sequence that has the form
\[H^s_c(G_n; \pi_t(E_n \wedge Y)) \Rightarrow \pi_{t-s}\bigl(L_{K(n)}(Y)\bigr),\] 
where the $E_2$-term is continuous cohomology, $\pi_t(E_n \wedge Y)$ 
is a continuous profinite $G_n$-module (this 
structure is induced by $G_n$ acting diagonally on $E_n \wedge Y$, with $Y$ given the trivial $G_n$-action), 
and $L_{K(n)}(-)$ denotes Bousfield localization with respect to 
$K(n)$ (this result is due to \cite[Corollary 8.2.4, Theorem 8.2.5]{behrensdavis} and 
\cite[Theorem 1]{devinatzhopkins}; see 
also \cite[Proposition 7.4]{homasa}). Therefore, 
to make sense of $E_n$ as a {\em continuous} $G_n$-spectrum is a fundamental problem.

For the closed subgroups $G$ of $G_n$, Devinatz and Hopkins \cite{devinatzhopkins} gave a construction of commutative $S$-algebras, here denoted by $E_n^{dhG}$, that behave like continuous homotopy fixed point spectra. However, the construction of the $E_n^{dhG}$ does not make use of 
a continuous action of $G$ on $E_n$. Using the construction of $E_n^{dhU}$ for open normal subgroups $U$ of $G_n$, a new and systematic definition of homotopy fixed points with respect to a continuous $G$-action, for arbitrary closed subgroups $G$, was given in \cite{davis}: we denote these continuous homotopy fixed points by $E_n^{h'G}$. The formation of the $E_n^{h'G}$ 
is based on the notion of discrete $G$-spectrum (a spectrum that is built out of simplicial discrete $G$-sets) and homotopy limits of towers of discrete $G$-spectra (such homotopy limits are the continuous $G$-spectra of \cite{davis}).

In \cite{hfplt}, a different construction for a continuous homotopy fixed point spectrum $E_n^{hG}$ and its descent spectral sequence, independent of \cite{devinatzhopkins} and \cite{davis}, has been obtained. The approach of \cite{hfplt} is to consider $E_n$ as an object in the category of profinite $G$-spectra, 
which, in contrast to the discrete $G$-spectra mentioned above, are $G$-spectra 
that are built out of simplicial profinite $G$-sets. In this framework, the profinite $G$-spectrum $E_n$ is a (homotopy) limit of certain spectra that 
are simultaneously discrete $G$-spectra and profinite 
$G$-spectra.  

Each one of the above approaches has its own advantages (and drawbacks). But, as one might expect, there are equivalences 
\[
E_n^{dhG} \simeq E_n^{h'G} \simeq E_n^{hG}
\] 
for every closed subgroup $G$, by 
\cite[Theorem 8.2.1]{behrensdavis} and \cite[top of page 220]{hfplt}, respectively.

It is no surprise that for the profinite $G_n$-spectrum $E_n$ there are 
iteration issues related to those mentioned just before Definition \ref{vcd}. For example, in \cite{devinatzhopkins} 
Devinatz and Hopkins did not obtain a construction of a 
continuous homotopy fixed point spectrum $(E_n^{dhK})^{hG/K}$  
when $G/K$ is not finite. Nevertheless, by a sophisticated study of the structure of $E_n^{dhK}$ as a $E_n^{dhG}$-module, Devinatz \cite{devinatz2} was able to construct 
a strongly convergent (Adams-type) Lyndon-Hochschild-Serre spectral sequence 
\begin{equation}\label{lhsSS}
H_c^*(G/K;\pi_\ast(E_{n}^{dhK})) \Rightarrow \pi_\ast(E_n^{dhG}),
\end{equation}
with $E_2$-term given by continuous cohomology.  

In \cite{davis2}, the first author was able to make sense of $E_n^{h'K}$ (as defined in \cite{davis}) as a continuous $G/K$-spectrum for an arbitrary closed normal subgroup $K$. Moreover, it was shown in \cite{davis2} that there is an equivalence 
\[E_n^{h'G}\simeq (E_n^{h'K})^{h'G/K}\] and a descent spectral sequence 
\begin{equation}\label{davis2SS}
H^\ast_c(G/K; \pi_\ast(E_n^{h'K})) \Rightarrow \pi_\ast(E_n^{h'G})
\end{equation}
that is
isomorphic to spectral sequence (\ref{lhsSS}), by \cite[Theorem 7.6]{davis2}. 

Though it is somewhat of an oversimplification, let us describe the results of 
the preceding paragraph as taking place in the ``world of continuous $G$-spectra" (this 
terminology is an adaptation of the ``$G$-world" terminology of \cite{jardine} (for example, see \cite[page 211]{jardine})). Then 
one of the purposes of this paper is to show that analogous results 
hold in the setting of profinite $G$-spectra, by applying results that were described in Section \ref{oneone} and by using the independent construction of continuous homotopy fixed points in \cite{hfplt}. 
We accomplish this goal by obtaining the following two results, Theorems 
\ref{iteratedthmEnintro} and \ref{thm:SS}.

\begin{theorem}\label{iteratedthmEnintro}
Let $G$ be an arbitrary closed subgroup of $G_n$ and let $K$ be a closed normal subgroup of $G$. The continuous homotopy fixed point spectrum $E_n^{hK}$ has a model in the category of profinite $G/K$-spectra, there is an iterated continuous homotopy fixed point spectrum $(E_n^{hK})^{hG/K}$ and there is an equivalence  
\[
E_n^{hG} \simeq (E_n^{hK})^{hG/K}. 
\] 
\end{theorem}  

The proof of Theorem \ref{iteratedthmEnintro} is given in 
Section \ref{LT} and the helpful notion of ``model" that is 
used in the theorem is explained in a precise 
way in Definition \ref{promodel}.

Before stating Theorem \ref{thm:SS}, we would like to be more explicit about the first two conclusions of Theorem 
\ref{iteratedthmEnintro} and thereby quickly illustrate that profinite iteration 
problems are not easy to solve. The proof of Theorem \ref{iteratedthmEnintro} 
shows that the spectrum $E_n^{hK}$ can be identified with the profinite 
$G/K$-spectrum 
\begin{equation}\label{long}
\holim_{k \geq 0} \holim_{q \in \mathbb{Z}} F^s_{G/K}\Bigl(\mspace{2mu}\colim_{U \vartriangleleft_o G} 
\mathrm{Map}(EG, P^q E'_{n, I_k})^{KU}\Bigr),\end{equation} so that 
\[(E_n^{hK})^{hG/K} = \Bigl(\holim_{k \geq 0} \holim_{q \in \mathbb{Z}} F^s_{G/K}\Bigl(\mspace{2mu}\colim_{U \vartriangleleft_o G} 
\mathrm{Map}(EG, P^q E'_{n, I_k})^{KU}\Bigr)\Bigr)^{\mspace{-3.5mu} hG/K}.\] In 
expression (\ref{long}), all the undefined notation is carefully explained in later sections, 
but to gain a fairly complete understanding of what spectrum (\ref{long}) is describing, 
it suffices to say that (a) in each of its applications above (and as mentioned 
just after Theorem \ref{firstmain}), 
$F^s_{G/K}(-)$ returns a profinite $G/K$-spectrum that is 
weakly equivalent to the $G/K$-spectrum that is its input and 
(b) morally, 
$P^q \Enik$ is ``the $q$th Postnikov section of the $G_n$-spectrum $E_n/I_k$," 
where 
\[\pi_\ast(E_n) \cong \lim_{k \geq 0} \pi_\ast(E_n)/I_k.\]

\begin{theorem}\label{thm:SS} 
Let $G$ and $K$ be as in Theorem \ref{iteratedthmEnintro}. 
There is a strongly convergent spectral sequence for iterated continuous homotopy fixed points
\[H_c^s(G/K; \pi_t(E_n^{hK})) \Rightarrow \pi_{t-s}(E_n^{hG}),\]
with $E_2$-term equal to the continuous cohomology of $G/K$ with coefficients 
the profinite $G/K$-module $\pi_t(E_n^{hK})$. 
This spectral sequence is isomorphic to the spectral sequences of 
$\mathrm{(\ref{lhsSS})}\negthinspace$ and 
$\mathrm{(\ref{davis2SS})}$, from the $E_2$-term 
onward. 
\end{theorem} 

The proof of Theorem \ref{thm:SS} is in Section \ref{LT}. Though this theorem is clearly the result that one desires 
to more fully tie together $E_n^{dhG}$, $E_n^{hG}$ and $E_n^{h'G}$, its proof is 
quite intricate and a (very) brief road map might be useful: the proof can be described as
consisting of a chain of isomorphisms between spectral sequences.

\subsection{The importance of iterated homotopy fixed points in chromatic 
homotopy theory}
It is worth noting that iterated continuous homotopy fixed points for $E_n$ are 
not just of purely theoretical interest. For example, \cite[page 2883]{davis2} (building on 
\cite[page 133]{homotopydevinatz}) shows that 
certain instances of 
$(E_n^{h'K})^{h'G/K}$ play a useful role in the work of Devinatz \cite{homotopydevinatz, 
devinatzfinite} on the major 
conjecture in chromatic homotopy theory that $\pi_\ast\bigl(L_{K(n)}(\mathbb{S}^0)\bigr)$ 
is a module of finite type over $\mathbb{Z}_p$. Also, given a continuous epimorphism 
$G_n \to \mathbb{Z}_p$ of groups with kernel $K$ and the topological generator 
$1$ of $\mathbb{Z}_p$, \cite[Proposition 8.1]{devinatzhopkins} shows that 
a model for $(E_n^{dhK})^{h\mathbb{Z}_p}$, the 
continuous $\mathbb{Z}_p$-homotopy fixed points of $E_n^{dhK}$, is given 
by taking the homotopy fiber of the map \[E_n^{dhK} \xrightarrow{\mathrm{id}-1} E_n^{dhK}\] 
(this construction of the continuous $\mathbb{Z}_p$-homotopy fixed points is a special case 
of a well-known technique (for example, see \cite[\S 2.2]{ghm})), and this 
homotopy fiber sequence plays a role in constructing an interesting element in 
$\pi_{-1}\bigl(L_{K(n)}(\mathbb{S}^0)\bigr)$, for all $n$ and $p$ \cite[Theorem 6]{devinatzhopkins}. 

Other examples of the importance 
of $(E_n^{h'K})^{h'G/K}$ occur 
in \cite[end of \S 1.1]{westerland} and \cite[\S 5.5]{westerland}. In the 
last reference, a doubly iterated homotopy fixed point spectrum 
\[\Bigl((E_n^{h'K})^{\mspace{1mu}h'G/K}\Bigr)^{\mspace{-4mu}h'G_n/G}\] makes an appearance (we refer the 
reader to \cite[\S 2.2, Corollary 3.25, \S 5.5]{westerland} for the definitions of $K$ and $G$). Given 
these examples, we expect there to be situations where 
$(E_n^{hK})^{hG/K}$ will be a useful tool in chromatic theory.

\subsection{A technical advantage of our proof of Theorem \ref{iteratedthmEnintro} 
that is not possessed by the Davis and Devinatz-Hopkins proofs.}
We close our introduction by pointing out a subtle feature of our proof that 
$(E_n^{hK})^{hG/K} \simeq E_n^{hG}$ is always valid (Theorem \ref{iteratedthmEnintro}), a feature 
that is not enjoyed (i) by the proof in \cite{davis2} that 
$(E_n^{h'K})^{h'G/K} \simeq E_n^{h'G}$ always holds, or (ii) by the 
proof in \cite{devinatzhopkins} that when $G/K$ is finite, $(E_n^{dhK})^{hG/K} \simeq 
E_n^{dhG}$ (see \cite[Theorem 4, proof of Proposition 7.1]{devinatzhopkins}). To 
see this subtlety, we begin by noting that \cite[Corollary 5.5]{devinatzhopkins} implies that 
if $U$ is an open subgroup of $G_n$, then there is an equivalence
\begin{equation}\label{technical}
L_{K(n)}(E_n^{dhU} \wedge E_n) \simeq \textstyle{\prod}_{_{G_n/U}} E_n,
\end{equation} 
where the right-hand side is a finite product of $|G_n/U|$ copies of 
$E_n$. For our purposes here, we want to point out that the proof of (\ref{technical}) 
(see \cite[pages 24--30]{devinatzhopkins}) is highly nontrivial and, in particular, it 
uses (see \cite[page 29: the proof of Theorem 5.3]{devinatzhopkins}) the deep result due to Hopkins and Ravenel 
that there exists a finite spectrum $W$ with torsion-free $\mathbb{Z}_{(p)}$-homology, 
such that the continuous cohomology groups $H^{s, \ast}_c(U; \pi_\ast(E_n \wedge W)/I_n\pi_\ast(E_n \wedge W))$ vanish for all $s$ greater than some $s_0$ (see \cite[Lemmas 8.3.5--8.3.7]{ravenelorange}). It is worth noting that this result of Hopkins and Ravenel played a key role in 
the proof of the very important smashing conjecture (which states 
that $L_{E(n)}(-)$ is a smashing localization, where $E(n)$ is the Johnson-Wilson spectrum; 
see \cite[Theorem 7.5.6, Chapter 8]{ravenelorange}).

(To keep our explanation of the above point from being too long, the rest of our 
discussion is written in a style that assumes the 
reader has certain portions of \cite{davis2, devinatzhopkins} readily available.) 
  
The proof in \cite{devinatzhopkins} of the iteration result in (ii) above depends on 
(\ref{technical}). (The details for this assertion are as follows: \cite[proof of 
Proposition 7.1]{devinatzhopkins} uses the isomorphism in \cite[(6.5)]{devinatzhopkins}, 
and the proof of this isomorphism 
uses (\ref{technical}) (see the second equality in 
\cite[proof of Proposition 6.3]{devinatzhopkins}).) 
Similarly, 
the proof of the iteration 
result in (i) applies (\ref{technical}). (In detail: the result in (i) is \cite[Theorem 7.3]{davis2}; its proof depends on \cite[Lemma 7.1]{davis2}; the proof of this lemma uses 
\cite[Proposition 6.3]{devinatzhopkins}; and as noted above, the proof of this 
proposition uses (\ref{technical}).) 

However, the proof of Theorem \ref{iteratedthmEnintro} avoids using (\ref{technical}): 
this is not hard to see by going over the proof and by noting that in its use of the proof of 
Theorem \ref{fiterated} (see the proof of Theorem \ref{iteratedthm}), it is able to 
utilize \cite[proof of Lemma 4.9]{davis2}, thanks to the use of Postnikov towers (the properties 
of Postnikov towers that are relevant here are discussed in the first paragraph of 
Section \ref{usedinintro}). By contrast, the proof of the result in (i) (that is, \cite[proof of Theorem 7.3]{davis2}) also used \cite[proof of Lemma 4.9]{davis2}, but to do so, it needed to 
apply (\ref{technical}), as described above. 
Thus, the proof of well-behaved iteration for $E_n$ in the 
profinite setting has the interesting technical advantage that it is simpler than the proofs 
referred to in 
(i) and (ii), in that the proof in the profinite setting does not depend on the deep result of 
Hopkins and Ravenel. 

\subsection*{Acknowledgements} 
We thank Ethan Devinatz for helpful remarks about the proof of Corollary 
\ref{goodtool}, and the referee of the first version of this paper for helpful comments, including remarks that led to improvements in the paper's presentation. Also, we thank the referee of a revised 
version of this paper for useful comments, including one that 
simplified the argument in the paragraph after Remark \ref{wepausehere}.

\section{Spectra with a continuous $G$-action}
Everywhere in this paper, unless stated otherwise, $G$ denotes a profinite group. In this section we quickly review the 
categories of spectra with continuous $G$-action that we need for our work.
\subsection{Discrete $G$-spectra}
We summarize the most important properties of simplicial discrete $G$-sets and discrete $G$-spectra. More details can be found in \cite{goerss}, \cite{davis} and \cite{behrensdavis}.

A $G$-set $S$ is called discrete if the action is continuous when $S$ is given the discrete topology. This is equivalent to requiring that the stabilizer of any element in $S$ be an open subgroup in $G$ and to asking that $S$ be equal to the colimit of fixed points
\[
S=\colim_U S^U
\]
over the open subgroups $U$ of $G$. A simplicial discrete $G$-set is a simplicial object in the category of discrete $G$-sets. By defining morphisms as levelwise $G$-equivariant maps we obtain the category of simplicial discrete $G$-sets which we denote by $\Shg$.

In \cite[Theorem 1.12]{goerss}, Goerss showed that there is a model structure on $\Shg$ for which the cofibrations are the monomorphisms and the weak equivalences are the morphisms whose underlying maps of simplicial sets are weak equivalences in the standard model structure on the category $\Sh$ of simplicial sets. The category $\Shpg$ of pointed simplicial discrete $G$-sets inherits a model structure from $\Shg$ in the usual way: a map is a weak equivalence (cofibration, fibration) if its underlying map in $\Shg$ is a weak equivalence (cofibration, fibration, respectively). 

In order to stabilize the category $\Shpg$ we consider the category $\Spg$ of discrete $G$-spectra. An object $X$ of $\Spg$ consists of a sequence $\{X_n\}_{n\geq0}$ of pointed simplicial discrete $G$-sets $X_n$ together with structure maps
$$\sigma_n:S^1\wedge X_n \to X_{n+1}$$
in $\Shpg$, where we consider $S^1$ as a pointed simplicial discrete $G$-set with trivial $G$-action. A map $f:X\to Y$ of discrete $G$-spectra is a sequence of maps $f_n:X_n \to Y_n$ in $\Shpg$ which are compatible with the structure maps.  

The model structure on $\Shpg$ is left proper and cellular: see \cite{goerss} or \cite[Lemma 2.1.3]{behrensdavis}. This allows one to use Hovey's stabilization methods \cite{hovey} in order to prove the following theorem, as in \cite[Theorem 2.2.1]{behrensdavis}; the initial proof using presheaves of spectra was given in \cite{davis}.

\begin{theorem}
The category $\Spg$ admits a model structure in which a map is a weak equivalence (cofibration) if its underlying map of Bousfield-Friedlander spectra is a weak equivalence (cofibration). 
\end{theorem}

\begin{remark}\label{picont}
For a discrete $G$-spectrum $X$, there is an induced action of $G$ on $\pi_*X$, such that each stable homotopy group $\pi_kX$ is a discrete $G$-module (see \cite[Corollary 3.12]{davis}). 
\end{remark}
\begin{remark} 
There is also a version of discrete $G$-spectra based on symmetric spectra: see \cite[\S 2.3]{behrensdavis}. But for the purposes of this paper it suffices to consider the model structure of the previous theorem on $\Spg$.
\end{remark}
\subsection{Mapping spectra I}\label{mappingone}
In order to study homotopy fixed points we will need the following notion of mapping spectrum. 
Let $T$ be any set. Then the set $\Map_c(G, T)$ of continuous functions $G \to T$, where $T$ is 
regarded as a space with the discrete topology, is 
a discrete $G$-set with $G$-action given by $(gf)(h)=f(hg)$. If $Y$ is a 
simplicial set, the mapping space $\Map_c(G,Y)$ is defined to be the simplicial discrete $G$-set given in degree $m$ by 
$$\Map_c(G,Y)_m =\Map_c(G,Y_m).$$

Now let $X$ be any spectrum. The continuous mapping spectrum 
$\Map_c(G,X)$ is defined to be the discrete $G$-spectrum whose $n$th space is  
$$\Map_c(G, X_n).$$ It is not hard to see that there is an isomorphism of spectra
\[\Map_c(G, X) \cong \colim_{N \vartriangleleft_o G} \textstyle{\prod}_{G/N} X,\] where the 
colimit is over the open normal subgroups of $G$. Also, 
if $X$ is a discrete $G$-spectrum, we again write $\Map_c(G, X)$ for the continuous 
mapping spectrum that is obtained as above, by just regarding $X$ as a spectrum. 

\subsection{Profinite $G$-spectra} 
A profinite space is a simplicial object in the category $\hEh$ of profinite sets. Together with levelwise continuous maps, profinite spaces form a category that is denoted by $\hSh$. If $\Sh$ denotes the category of simplicial sets, then the forgetful functor $|\cdot|:\hSh \to \Sh$ has a left adjoint $\compl:\Sh \to \hSh$ which we call profinite completion. There is a model structure on $\hSh$ for which the cofibrations are the monomorphisms and a weak equivalence is a map $f$ which induces isomorphisms on $\pi_0$, the profinite fundamental group and on continuous cohomology with finite local coefficients. We refer to \cite{gspaces} and \cite{gspectra} for the details. 
 
Let $S$ be a profinite set with a continuous map $\mu:G \times S \to S$ that satisfies the axioms of a group action. We call such an $S$ a profinite $G$-set. If $X$ is a profinite space and $G$ acts continuously on each $X_n$ such that the action is compatible with the structure maps, then we call $X$ a profinite $G$-space. We use $\hShg$ to denote the category of profinite $G$-spaces with $G$-equivariant maps of profinite spaces as morphisms. If $X$ is a pointed profinite space with a continuous $G$-action that fixes the basepoint, then we call $X$ a pointed profinite $G$-space. We denote the corresponding category by $\hShpg$. Also, we let $\hShp$ denote the 
category of pointed profinite spaces.

The category $\hShpg$ carries a fibrantly generated left proper simplicial model structure for which a map $f:X \to Y$ is a weak equivalence if and only if its underlying map is a weak equivalence in $\hShp$ and is a cofibration if and only if $f$ is a levelwise injection and the action of $G$ on $Y_n -f(X_n)$ is free for each $n \geq 0$. The corresponding homotopy category is denoted by $\hHhpg$. 

We would like to stabilize the category of pointed profinite spaces. Since the simplicial circle $S^1=\Delta^1/\partial \Delta^1$ is a simplicial finite set and hence an object in $\hShp$, we may stabilize $\hShp$ by considering sequences of pointed profinite spaces together with bonding maps for the suspension. 
In more detail, a profinite spectrum $X$ consists of a sequence of pointed profinite spaces $X_n \in \hShp$ and maps $\sigma_n:S^1 \wedge X_n \to X_{n+1}$ in $\hShp$ for $n\geq0$. A morphism $f:X \to Y$ of spectra consists of maps $f_n:X_n \to Y_n$ in $\hShp$ for $n\geq0$ such that $\sigma_n(1\wedge f_n)=f_{n+1}\sigma_n$. We denote by $\hSp$ the corresponding category of profinite spectra.
 
There is a stable simplicial model structure on the category $\hSp$. 
Also, the levelwise profinite completion functor is a left Quillen functor from Bousfield-Friedlander spectra $\Sp$ to $\hSp$.
  
A profinite $G$-spectrum $X$ is a sequence of pointed profinite $G$-spaces $\{X_n\}$ together with pointed $G$-equivariant maps $S^1\wedge X_n \to X_{n+1}$ for each $n\geq 0$, where $S^1$ is equipped with a trivial $G$-action. A map of profinite $G$-spectra $X \to Y$ is a collection of maps $X_n \to Y_n$ in $\hShpg$ compatible with the structure maps of $X$ and $Y$. 
The following theorem was proved in \cite{gspectra}.
\begin{theorem}\label{existsmodel}
There is a stable left proper simplicial model structure on $\hSpg$ in which a map between fibrant profinite $G$-spectra is an equivalence if and only if it is an equivalence in $\hSp$. 
\end{theorem}

In general, equivalences in $\hSpg$ do not have the nice characterization given in the above theorem, and we recall that an arbitrary morphism in $\hSpg$ is an equivalence exactly when any projective cofibrant replacement induces an equivalence of mapping spaces upon application of the functor \[\mathrm{map}_{\hSpg}(-, E) \colon \hSpg \to \Sh,\] 
for all $\Omega$-spectra $E$ in 
$\hSpg$ (we refer the reader to \cite{gspectra} for more detail). Also, to add a little to the use of equivalences in $\hSp$, in Theorem \ref{existsmodel}, we recall that an equivalence between fibrant objects in $\hSp$ is also an equivalence between fibrant objects in $\Sp$.

Finally, if $X$ is a profinite $G$-spectrum, then there is an induced action of $G$ on each stable profinite homotopy group $\pi_k(R_GX)$, where $R_G$ denotes a fibrant replacement functor for profinite $G$-spectra: this $G$-action is compatible with the profinite structure and each stable profinite homotopy group $\pi_k(R_GX)$ is a continuous profinite $G$-module. 
Thus, the topological $G$-module structure of $\pi_k(R_GX)$ reflects the 
character of $X$ as a profinite $G$-spectrum and, to ease our notation, we will write 
just $\pi_kX$ for this $G$-module. 
 
\subsection{Mapping spectra II}\label{htplimits} 

For a detailed discussion of continuous mapping spectra for profinite spectra and profinite $G$-spectra, we refer the reader to \cite{hfplt}. Here we summarize only the basic definitions.

For $X, Y \in \hShp$, the mapping space $\map_{\hShp}(X,Y)$ is defined to be the simplicial set whose set of $n$-simplices is given as the set of maps 
\[
\map_{\hShp}(X,Y)_n=\Hom_{\hShp}(X\wedge \Delta[n]_+, Y).
\] 
For $X, Y \in \hShpg$, the mapping space $\map_{\hShpg}(X,Y)$ is defined to be the simplicial set whose set of $n$-simplices is given as the set of maps 
\[
\map_{\hShpg}(X,Y)_n=\Hom_{\hShpg}(X\wedge \Delta[n]_+, Y)
\]
where $\Delta[n]_+$ is considered as a pointed profinite $G$-space with trivial $G$-action. 

Let $Y$ be a profinite space and $W$ be a pointed profinite space. The functor $\hSh \to \hShp$ that sends $Y$ to $Y_+$ (defined by adding a disjoint basepoint) is the left adjoint of the functor that forgets the basepoint. As in \cite{hfplt}, we will use the notation $\Map(Y,W)$ for the pointed simplicial set $\map_{\hShp}(Y_+,W)$ whose basepoint is the map $Y_+\to \ast \to W$. This defines a functor
\[
\Map(-,-): \hSh^{\op} \times \hShp \to \Shp.
\]

For a profinite space $Y$ and a profinite spectrum $X$, we denote by $\Map(Y, X)$ the spectrum whose $n$th space is given by the pointed simplicial set $\Map(Y,X_n)$. This defines a functor
$$\Map(-,-): \hSh^{\op} \times \hSp \to \Sp.$$

Now let $Y$ be a profinite $G$-space and let $W$ be a pointed profinite $G$-space. The pointed simplicial set $\Mapg(Y,W)$ is defined to be the pointed simplicial set $\map_{\hShpg}(Y_+,W)$ with basepoint equal to the map $Y_+\to \ast \to W$. This defines a functor
\[
\Mapg(-,-): \hShg^{\op} \times \hShpg \to \Shp.
\]

When $Y$ is a profinite $G$-space and $W$ is a pointed profinite $G$-space, we equip the pointed simplicial set $\Map(Y,W)$ with a $G$-action by defining $(gf)(y):= gf(g^{-1}y)$. With this $G$-action, $\Mapg(Y,W)$ is the pointed space of $G$-fixed points of the pointed  space $\Map(Y,W)$. 

If $Y$ is a profinite $G$-space and $X$ is a profinite $G$-spectrum, then $\Mapg(Y, X)$ is the spectrum whose $n$th space is given by the pointed simplicial set $\Mapg(Y,X_n)$. This construction yields a functor
\[
\Mapg(-,-): \hShg^{\op} \times \hSpg \to \Sp.
\]

The reader will have noticed that there are various mapping spectra ``out of a profinite group" appearing in the theories of discrete and profinite $G$-spectra. In fact, besides the constructions \[\mathrm{Map}_H(EG,X) = \mathrm{Map}(EG,X)^H\] for 
a profinite $G$-spectrum $X$ and a closed subgroup $H$ in $G$ (this mapping spectrum was seen in (\ref{long}); 
$EG$ is defined in Section \ref{rghg}) and 
$\mathrm{Map}_c(G,Z)$ for any spectrum $Z$, there is also the discrete $G$-spectrum $\mathrm{Map}_c^\ell(G, Z)$ -- its underlying spectrum is 
the same as that of $\mathrm{Map}_c(G, Z)$, but its $G$-action is induced by the $G$-action that 
is defined on each of the sets 
\[\mathrm{Map}_c(G,Z_n)_m = \mathrm{Map}_c(G, (Z_n)_m)\] 
by $(gf)(h) = f(g^{-1}h)$ (the notation ``$\mathrm{Map}_c^\ell(G, Z)$" follows that of 
\cite[page 328]{davis}; the motivation for this construction comes from \cite[Theorem 1(iii), Warning 4.19]{devinatzhopkins}) -- and the symmetric spectrum $\mathrm{Map}^c(G,Y)$, 
where $Y$ is any symmetric spectrum, that is defined in \cite[Section 2.4]{behrensdavis} (and used throughout \cite{behrensdavis}). 
Some of the relationships between the various constructions in the mini-topic of continuous mapping spectra ``out of $G$" are considered 
in \cite[Section 3.4 and page 219]{hfplt}.
 
\section{Comparing continuous homotopy fixed points}

In this section we recall the definition of homotopy fixed points for each of discrete and profinite $G$-spectra and show that they agree in certain cases where they are both defined. Our recollections start with the profinite case, for which details can be found in \cite{hfplt}. As usual $G$ denotes a profinite group.

\subsection{Homotopy fixed points of profinite $G$-spectra}\label{rghg}

A very convenient feature of the profinite approach is that the universal classifying space $EG$ of our profinite group (given, as usual, in degree $n$ by $G^{n+1}$) is naturally a profinite $G$-space. Thus, for a profinite $G$-spectrum $X$ it is possible to form the continuous mapping spectrum $\Mapg(EG,X)$. Moreover, $EG$ is a cofibrant profinite $G$-space, since $G$ acts freely in each degree, and hence, we can consider $EG$ as a cofibrant resolution of a point in $\hShg$. If $X$ is a fibrant profinite $G$-spectrum, then $\Mapg(EG,X)$ is a fibrant spectrum, giving a homotopically well-behaved version of the fixed points $\Mapg(\{\ast\},X)$. Thus, we let $R_G$ denote a fibrant replacement functor in $\hSpg$. In \cite{hfplt}, for any profinite $G$-spectrum $X$, the continuous homotopy fixed points of $X$ under $G$ were defined to be
\begin{equation}\label{profinhfp}
X^{hG}:= \Mapg(EG, R_GX),
\end{equation} 
and it was shown that if $X \to Y$ is an equivalence in $\hSpg$, then the induced map 
$X^{hG} \to Y^{hG}$ is an equivalence between $\Omega$-spectra in $\mathrm{Sp}(\mathcal{S}_\ast)$.

One advantage of the construction in (\ref{profinhfp}) is that the associated descent spectral sequence arises naturally from the filtration of $EG$ just as in the classical case for finite groups. 
This descent spectral sequence has the form 
\begin{equation}\label{hfpss}
E_2^{s,t}=H_c^s(G;\pi_tX) \Rightarrow \pi_{t-s}X^{hG}
\end{equation}
where the $E_2$-term is the continuous cohomology of $G$ with coefficients in the profinite $G$-module $\pi_tX$. 

One way to describe the above spectral sequence is as follows. Let $X$ be a fibrant profinite $G$-spectrum. We can consider $\Mapg(EG,R_GX)$ as coming from a cosimplicial spectrum $\Mapg(G^{\bullet+1},X)$ whose $n$th spectrum is $\Mapg(G^{n+1},X)$, where here, $G^{n+1}$ is the constant simplicial profinite $G$-space associated to the profinite group $G^{n+1}$. Then there is an equivalence
\[
X^{hG} \simeq \holim_{\Delta} \Mapg(G^{\bullet+1},X)
\]
and spectral sequence \eqref{hfpss} is isomorphic to the spectral sequence associated to the tower of spectra 
\[
\{\mathrm{Tot}_k\,(\Mapg(G^{\bullet+1},X))\}_k.
\] 
We refer the reader to \cite[\S 3]{hfplt} for the proofs and more details.

\subsection{Homotopy fixed points of discrete $G$-spectra}
For discrete $G$-spectra, the bad news is that, in general, $EG$ is not a simplicial discrete $G$-set. But the good news (for example, see \cite[Lemma 2.3, Corollary 2.4]{goerss}) is that a one-point space is a cofibrant object in $\Shg$. Thus, instead of using $EG$, the homotopy fixed points of a discrete $G$-spectrum $X$ are defined as the fixed points of a fibrant replacement $X_{f,G}$ in $\Spg$. As 
stated in the introduction, we denote these homotopy fixed points by $X^{h_dG}$ in order to distinguish them from the previous construction. Thus, as in \cite{davis}, we set
\[
X^{h_dG}:=(X_{f,G})^G.
\] 
A nice feature of this definition is that it is clear that homotopy fixed points are the right derived functor of the right Quillen functor $(-)^G: \Spg \to \Sp$, 
so that if $k \colon X \to Y$ is a 
weak equivalence of discrete $G$-spectra, then the map 
\[k^{h_dG} \colon X^{h_dG} = (X_{f,G})^G \xrightarrow{\,\simeq\,} (Y_{f,G})^G = Y^{h_dG}\] 
is an equivalence of spectra. In the definition of $k^{h_dG}$ above, $X_{f,G}$ and $Y_{f,G}$ are the output of a fibrant replacement functor 
$(-)_{f,G} \colon \mathrm{Sp}(\mathcal{S}_{\ast G}) \to 
\mathrm{Sp}(\mathcal{S}_{\ast G})$ 
for the model category of discrete $G$-spectra.

To obtain a descent spectral sequence in this setting, it is convenient to consider a description of homotopy fixed points that is different from the above definition. Let $\Gamma_G(-)$ be the endofunctor 
\[
X\mapsto \Map_c(G, X) =: \Gamma_G(X)
\] 
on discrete $G$-spectra (the object $\Map_c(G,X)$ is defined in Section \ref{mappingone}). 
The iterated application of $\Gamma_G(-)$ defines a cosimplicial object 
\[
(\Gamma_GX)^{\bullet} = \Map_c(G^{\bullet+1}, X)
\] in discrete $G$-spectra, with 
\[
(\Gamma_GX)^j\cong \Map_c(G^{j+1},X),
\] 
for each $j \geq 0$.

Whenever $G$ has finite virtual cohomological dimension and 
$X \overset{\simeq}{\longrightarrow} Y$ is a weak equivalence in $\Spg$ with 
$Y$ fibrant as a spectrum, it follows from \cite{davis} that there is an equivalence
\begin{equation}\label{dishfpmodel}
X^{h_dG} \simeq \holim_{\Delta} \Map_c(G^{\bullet+1}, Y)^G
\end{equation}
and a descent spectral sequence 
\begin{equation}\label{dishfpss}
E_2^{s,t}=H_c^s(G; \pi_tX) \Rightarrow \pi_{t-s}X^{h_dG}
\end{equation}
whose $E_2$-term is the continuous cohomology of $G$ with coefficients in the discrete $G$-module $\pi_tX$.
  
%
%%%%%%%%%%%%%%%%%%%%%%%%%%%%%%%%%%%%%%%
%
\subsection{Comparison of homotopy fixed points}
In this section, we show under mild assumptions on the profinite group $G$ that the two notions of continuous homotopy fixed points coincide in several situations in which they are both defined. 

\begin{defn}\label{fgspectrum}
A fibrant profinite $G$-spectrum $X$ is called an $f$-$G$-spectrum if each space $X_n$ is a simplicial finite discrete $G$-set.
\end{defn}

Let $X$ be an $f$-$G$-spectrum. Since $X$ is fibrant as a profinite $G$-spectrum, 
the homotopy groups of $X$ are all finite discrete $G$-modules, by \cite[proof of 
Proposition~3.9]{gspectra}. Thus, an $f$-$G$-spectrum is 
an $f$-spectrum in the sense of \cite[page~5]{browncomenetz} (that is, each homotopy group of $X$ is finite), which explains part of the motivation for the terminology of Definition \ref{fgspectrum}. Since $X$ is both a profinite and a discrete $G$-spectrum, we have our two different notions of continuous homotopy fixed points at hand. 
\begin{theorem}\label{profdisc}
Let $G$ be a profinite group with finite virtual cohomological dimension and let $X$ be an $f$-$G$-spectrum. There is an equivalence of spectra 
\begin{equation}\label{compmap}
X^{hG}\simeq X^{h_dG}.
\end{equation}
\end{theorem}
\begin{proof}
By \cite{davis}, since $X$ is fibrant as a spectrum and $G$ has finite virtual cohomological dimension, we can use the homotopy limit $\holim_{\Delta} \Map_c(G^{\bullet+1}, X)^G$ as a model for  $X^{h_dG}$. 
There is an isomorphism of cosimplicial spectra 
\begin{equation}\label{cosimp}
\Map_G(G^{\bullet+1}, X) \cong \Map_c(G^{\bullet+1}, X)^G
\end{equation} 
(see \cite[page 219 (especially (16)); proof of Lemma 4.9]{hfplt}). 
By \cite[Proposition 3.23]{hfplt}, 
this shows that we have equivalences 
\[
X^{hG} \simeq \holim_{\Delta} \Mapg(G^{\bullet+1}, X) \simeq \holim_{\Delta} \Map_c(G^{\bullet+1}, X)^G.
\] 
Hence there is an equivalence of spectra 
$X^{hG} \simeq X^{h_dG},$ as desired. 
\end{proof}

\begin{remark}\label{arbitrary}
Let $X$ be any $f$-$G$-spectrum. It is worth noting that even when $G$ does 
not have finite virtual cohomological dimension, \cite[Theorem 3.5]{davisfibrant} shows that 
the spectrum ${\displaystyle{\colim_{N \vartriangleleft_o G} \bigl(\holim_\Delta 
\mathrm{Map}_c(G^{\bullet+1}, X)\bigr)^{\mspace{-2mu}N}}}$, a colimit over the open normal subgroups of $G$, is a fibrant discrete $G$-spectrum, and 
hence,
\begin{align*}X^{hG} & \simeq \holim_\Delta \mathrm{Map}_c(G^{\bullet+1}, X)^G 
\\ & \cong \Bigl(\,\colim_{N \vartriangleleft_o G} \bigl(\holim_\Delta 
\mathrm{Map}_c(G^{\bullet+1}, X)\bigr)^{\mspace{-2mu}N}\Bigr)^{\mspace{-3mu}G} \\ & \simeq 
\Bigl(\,\colim_{N \vartriangleleft_o G} \bigl(\holim_\Delta 
\mathrm{Map}_c(G^{\bullet+1}, X)\bigr)^{\mspace{-2mu}N}\Bigr)^{\mspace{-3mu}h_dG},\end{align*} 
so that $X^{hG}$ can always be regarded as being the $G$-homotopy 
fixed points of some discrete $G$-spectrum.
\end{remark}

After a few preparatory comments, we recall a theorem that gives an example of a way that $f$-$G$-spectra arise. We call a spectrum $X\in \Sp$ a $G$-spectrum (without taking any topology into account) if each space $X_n$ is a pointed $G$-space and the $G$-actions are compatible with the bonding maps $S^1\wedge X_n \to X_{n+1}$. 

\begin{defn}\label{pifinitedef}
A $G$-spectrum $Z$ is called $\pi$-finite if all its homotopy groups are finite.  
\end{defn}

Let $X$ be an arbitrary $f$-$G$-spectrum: since 
the homotopy groups of $X$ are all finite discrete $G$-modules, the underlying $G$-spectrum of $X$ is $\pi$-finite. With respect to this conclusion, the following converse was proved in \cite{gspectra}, Theorem 5.15, for the case when $G$ is strongly complete (that is, if every subgroup of finite index is open in $G$).
\begin{theorem}\label{Gstablefinitecompletion}
Let $G$ be a strongly complete profinite group and let $X$ be a $\pi$-finite $G$-spectrum. Then there is a $G$-equivariant map 
$$\varphi^s: X \to F^s_GX$$ 
of spectra from $X$ to an $f$-$G$-spectrum $F^s_GX$ such that $\varphi^s$ is a stable equivalence of underlying spectra.  

The assignment $X\mapsto F_G^sX$ is functorial in the sense that given a $G$-equivariant map $h:X\to Y$ between $\pi$-finite $G$-spectra, there is a map $F_G^s(h)$ in $\hSpg$ such that the 
diagram 
$$\xymatrix{
X\ar[d] \ar[r]^h & Y \ar[d] \\
F_G^sX \ar[r]_{F_G^s(h)} & F_G^sY}$$ 
of underlying spectra commutes.
\end{theorem}

In Theorem \ref{Gstablefinitecompletion}, if the map $h$ is a stable 
equivalence of spectra, then the map $F^s_G(h)$ is too, and hence, 
$F^s_G(h)$ is a weak equivalence of discrete $G$-spectra. However, 
when the map $h$ is a stable equivalence of spectra, it is not 
known that $F^s_G(h)$ is a weak equivalence of profinite $G$-spectra 
(nevertheless, there is still the equivalence in (\ref{itturnsoutyes!}) below).
                            
Theorem \ref{Gstablefinitecompletion} 
motivates the following definition.
\begin{defn}\label{twodefs}
Let $G$ be a strongly complete profinite group and let $X$ be a $\pi$-finite $G$-spectrum. 
Then both $(F^s_GX)^{hG}$ and $(F^s_GX)^{h_dG}$ can be formed and it is natural to 
define
\[X^{hG}:= (F^s_GX)^{hG}\] and 
\[X^{h_dG}:= (F^s_GX)^{h_dG}.\] 
\end{defn}

Each of the two notions of homotopy fixed 
points in Definition \ref{twodefs} 
have a homotopy invariance property: whenever $G$ is a 
strongly complete profinite group and $h \colon X \to Y$ 
is a $G$-equivariant map between $\pi$-finite $G$-spectra that 
is a stable equivalence of spectra, then the map
\begin{equation}\label{itturnsoutyes!}
(F^s_G(h))^{hG} \colon X^{hG} = (F^s_GX)^{hG} \xrightarrow{\,\simeq\,} 
(F^s_GY)^{hG} = Y^{hG}
\end{equation} 
is an equivalence (the verification that this map 
is an equivalence is delayed until Remark 
\ref{given_later} since the argument depends on some material 
that is developed later) and, since 
$F^s_G(h)$ is a weak equivalence of discrete $G$-spectra, the map 
\[(F^s_G(h))^{h_dG} \colon 
X^{h_dG} = (F^s_GX)^{h_dG} \xrightarrow{\,\simeq\,} (F^s_GY)^{h_dG} = Y^{h_dG}\] is an equivalence. 

% the next sentence was tweaked a little bit
The following result, which is immediate from Theorem \ref{profdisc}, 
describes a case when the two notions of 
homotopy fixed points in Definition \ref{twodefs} 
are equivalent to each other.

\begin{cor}\label{gotitinthiscase}    
Let $G$ be a strongly complete profinite group with finite virtual cohomological 
dimension and let $X$ be a $\pi$-finite $G$-spectrum. Then there is an equivalence
\[
X^{hG} \simeq X^{h_dG}.
\] 
\end{cor}
  
% this comment-out material isn't needed and can be ignored
%However, if $G$ has finite virtual cohomological 
%dimension, then there is an equivalence $X^{hG} \simeq Y^{hG}$ that 
%is obtained by applying Corollary \ref{gotitinthiscase} twice: 
%\[X^{hG} \simeq X^{h_dG} = (F^s_GX)^{h_dG} \xrightarrow{\,\simeq\,} 
%(F^s_GY)^{h_dG} \simeq Y^{hG}.\] But it turns out more can be achieved than just this positive step: the hypothesis of finite virtual cohomological dimension can be dropped, so that the above 
%equivalence of spectra always exists.  

%
%%%%%%%%%%%%%%%%%%%%%%%%%%%%%%%%%%%%%%%
%
\section{Iterated homotopy fixed point spectra in the profinite setting}\label{four}
\subsection{Recollections of basic facts and the main problem}\label{fourone} 

We begin by recalling some material about iterated 
homotopy fixed points from \cite{hfplt}. Let $K$ be a closed subgroup of $G$ and let $N(K)$ be the normalizer of $K$ in $G$. Also, let $X$ be any fibrant profinite $G$-spectrum. The composition 
\begin{equation}\label{identification}
\mathrm{Map}_K(EG, X) \overset{\simeq}{\longrightarrow} \mathrm{Map}_K(EK, 
X) \overset{\simeq}{\longrightarrow} \mathrm{Map}_K(EK, R_KX) = X^{hK}
\end{equation}
of weak equivalences of spectra shows that it is natural to make the identification 
\[
X^{hK} = \mathrm{Map}_K(EG, X).
\] 
With this identification, it is clear that the profinite quotient group $N(K)/K$ acts on $X^{hK}$. Note that by setting $K=G$, our discussion shows that there 
is the useful identification 
\[
X^{hG} = \mathrm{Map}_G(EG, X).
\]  
 
To simplify our notation, we now assume that $K$ is a closed normal subgroup of $G$. There is a canonical map 
\[
X^{hG} \to X^{hK}
\] 
that is defined by 
\[
X^{hG} = \mathrm{Map}_G(EG, X) = \mathrm{Map}(EG, X)^G \to \mathrm{Map}(EG, X)^K = X^{hK},
\] 
and it is easy to see that this map factors into the identity map
\[
X^{hG} = \mathrm{Map}(EG, X)^G \overset{=}{\longrightarrow} \bigl(\mathrm{Map}(EG, X)^K\bigr)^{G/K} = (X^{hK})^{G/K}
\] 
followed by the natural inclusion 
$\bigl(\mathrm{Map}(EG, X)^K\bigr)^{G/K} \to X^{hK}$. 
When the $G/K$-spectrum $X^{hK}$ is a profinite $G/K$-spectrum, then 
composition of the above identity map with the canonical map
\[
(X^{hK})^{G/K} \to \mathrm(X^{hK})^{hG/K}
\] 
yields a map 
\begin{equation}\label{iteratedmap}
X^{hG} \to (X^{hK})^{hG/K}.
\end{equation} 
Thus, a natural problem in the theory of profinite $G$-spectra is to show that whenever 
$X^{hK}$ is a profinite $G/K$-spectrum, map (\ref{iteratedmap}) to 
the iterated continuous homotopy fixed points $(X^{hK})^{hG/K}$ is an equivalence. 

The first issue in studying map (\ref{iteratedmap}) 
is that $X^{hK}$ does not in general carry the structure of a {\em profinite} $G/K$-spectrum. 
This observation relies on the fact that the set of continuous maps between two profinite sets is in general not a profinite set itself. But there are interesting cases when it can be 
shown that (\ref{iteratedmap}) is an equivalence.

To understand the simplest case, we consider the situation when $K$ is open in $G$, so 
that the group $G/K$ is finite. 
Whether or not $X^{hK}$ is a profinite $G/K$-spectrum, since 
$G/K$ is a finite discrete space, the fibrant spectrum 
$X^{hK}$ is automatically a discrete $G/K$-spectrum, and hence, 
there is always the canonical map \[X^{hG} \to F(E(G/K)_+, X^{hK})^{G/K} = 
(X^{hK})^{h_dG/K}\] (as written above, $(X^{hK})^{h_dG/K}$ is equal  
to the ``usual" homotopy fixed point spectrum for the discrete 
group $G/K$; we remark that in 
forming this spectrum, no fibrant replacement of $X^{hK}$ is needed 
since $X^{hK}$ is already a fibrant spectrum) that is defined to be the composition of the aforementioned identity 
map with the canonical map 
$(X^{hK})^{G/K} \to (X^{hK})^{h_dG/K}$.
Then an adjunction argument suffices to prove that the map 
$X^{hG} \to (X^{hK})^{h_dG/K}$ 
is an equivalence (see \cite{hfplt}, \S 3.5). 

Furthermore, 
if $X^{hK}$ is a fibrant profinite $G/K$-spectrum, then 
the fibrant replacement map $X^{hK} \to R_{G/K} X^{hK}$ of profinite $G/K$-spectra, when regarded 
as a map of spectra, is a weak equivalence between fibrant 
objects that is $G/K$-equivariant, and hence, the map 
$(X^{hK})^{G/K} \to (X^{hK})^{h_dG/K}$ can be identified with the map
\[(X^{hK})^{G/K} \to \mathrm{Map}_{G/K}(E(G/K), R_{G/K}X^{hK}) = (X^{hK})^{hG/K}.\] 
Therefore, whenever $X^{hK}$ is a fibrant profinite $G/K$-spectrum, 
the weak equivalence $X^{hG} \overset{\simeq}{\longrightarrow} (X^{hK})^{h_dG/K}$ 
can be identified with map (\ref{iteratedmap}) to the iterated continuous 
homotopy fixed point spectrum $(X^{hK})^{hG/K}$, giving our first case of when (\ref{iteratedmap}) 
is a weak equivalence. 

\subsection{Some definitions and observations that help with understanding the problem of iteration}\label{fourtwo}

Now we give some preliminary 
considerations that lead to our second case of when map (\ref{iteratedmap}) is an equivalence. 
 
\begin{defn}\label{promodel}
Let $G$ be a profinite group and $X$ a $G$-spectrum. 
We say that $X$ has a model in the category of profinite $G$-spectra $\mathrm{(}$or that $X$ has a profinite $G$-model $X'\,\mathrm{)}$ 
if there exists a fibrant profinite $G$-spectrum $X'$ and a zigzag $\mathrm{(}$of any finite length$\mathrm{)}$ of $G$-equivariant morphisms of 
$G$-spectra between $X$ and $X'$, 
such that each underlying morphism of Bousfield-Friedlander spectra is a weak equivalence. For example, if the three arrows in the diagram \[Y_0 \xrightarrow{\,\simeq\,} Y_1 \xleftarrow{\,\simeq\,} 
Y_2 \xrightarrow {\,\simeq\,} Y_3\] denote $G$-equivariant maps between $G$-spectra, with each a weak equivalence of spectra, then $Y_0$ has a profinite $G$-model $Y_3$ when $Y_3$ is a fibrant profinite $G$-spectrum, and $Y_3$ has a 
profinite $G$-model $Y_0$ when $Y_0$ is a fibrant profinite $G$-spectrum.
\end{defn}
\begin{remark}
Let $G$ and $X$ be as in Definition \ref{promodel}. 
Since the homotopy groups of a fibrant profinite $G$-spectrum 
are profinite $G$-modules, 
when the homotopy groups $\pi_tX$ are not profinite $G$-modules, then the model described above cannot exist for $X$. Also, we point out that if $G$ is strongly 
complete and $X$ is $\pi$-finite, then the discrete $G$-spectrum 
$F^s_GX$ is a model for $X$ as a profinite $G$-spectrum.
\end{remark}
\begin{defn}\label{itdef}
Let $G$ be a profinite group and $K$ a closed normal subgroup of $G$. Let $X$ be a fibrant profinite $G$-spectrum such that $X^{hK}$ has a profinite $G/K$-model $X'(K)$.
Then we define the iterated continuous homotopy fixed points of $X$ to be 
\[
(X^{hK})^{hG/K}:=(X'(K))^{hG/K}.
\]
\end{defn}

\begin{remark}\label{zigzags}
Definitions \ref{promodel} and \ref{itdef} motivate the following question: if $G$ is a profinite group and $X$ is a $G$-spectrum that has $X'$ and $X''$ 
as profinite $G$-models, are $(X')^{hG}$ and $(X'')^{hG}$ weakly equivalent? 
Though we do not have an answer to this general question, we are able to 
handle it in the following special case: if $G$ is strongly complete and 
$X'$ and $X''$ are $f$-$G$-spectra (thus, $X$, $X'$, and $X''$ are 
$\pi$-finite $G$-spectra), 
then there is an equivalence \[(X')^{hG} \simeq (X'')^{hG}.\] For the subject 
matter of this paper, this assertion is quite pertinent: for example, it is applicable 
to situations that could occur in the settings of 
part (a) of Definition \ref{postnikovK} when $G/K$ is strongly complete, 
Theorem \ref{postnikovresult}, and Remark \ref{wepausehere}. 
It is easy to see that a proof of this special case follows from the technique used 
in the following justification of a very particular instance of it. Given 
$G$ strongly complete and $f$-$G$-spectra $X'$ and $X''$, we 
suppose additionally that $X'$ and $X''$ are models via the existence 
of the zigzag 
\[\xymatrix{X' & X \ar[l]^-{\,\simeq\,}_-{h_1} \ar[r]_-{\,\simeq\,}^-{h_2} & X''}\] 
of stable equivalences of spectra that are $G$-equivariant. It follows that 
there is the commutative diagram
\[\xymatrixcolsep{3pc}\xymatrix{
%R_GX' \ar[d]_-\simeq & 
X' \ar[d]^-\simeq_-{f_1} 
%\ar[l]_-\simeq 
& X \ar[d]^-\simeq_-{f_2} 
\ar[r]_-\simeq^-{h_2} \ar[l]^-\simeq_-{h_1} & X'' \ar[d]^-\simeq_-{f_3} %\ar[r]^-\simeq & F^s_G(X'') 
%\ar[d]^-\simeq 
\\ 
%R_G(F^s_GX')  & 
F^s_GX' 
% \ar[l]_-\simeq 
& F^s_GX   
\ar[r]_-\simeq^-{F^s_G(h_2)} \ar[l]^-\simeq_-{F^s_G(h_1)} & F^s_GX'' %\ar[r]^-\simeq & F^s_G(F^s_GX''),
\rlap{,}}\] in which the vertical maps are the usual ones obtained by 
applying $F^s_G$ and every map is $G$-equivariant and a stable equivalence 
of spectra. Since the four spectra in the corners of the diagram are 
$f$-$G$-spectra, they are discrete $G$-spectra, so that 
$f_1$ and $f_3$ are morphisms of profinite $G$-spectra. 
Therefore, to each of the four morphisms 
$f_1$, $f_3$, $F^s_G(h_1)$, and $F^s_G(h_2)$, we can 
apply the following fact: if $W \to Z$ is a map 
of profinite $G$-spectra and a stable equivalence of spectra, with $W$ and $Z$ 
$f$-$G$-spectra, then the induced map $W^{hG} \to Z^{hG}$ is a stable equivalence (for a justification of this fact, we refer the reader to the proof of a 
more general result in Remark 
\ref{moregeneral}). We find that there is the 
zigzag
\[(X')^{hG} \xrightarrow{\,\simeq\,} (F^s_GX')^{hG} \xleftarrow{\,\simeq\,} (F^s_GX)^{hG} \xrightarrow{\,\simeq\,} (F^s_GX'')^{hG} \xleftarrow{\,\simeq\,}(X'')^{hG}\] of equivalences, which yields the desired conclusion.
\end{remark}

As at the beginning of \S \ref{fourone}, we let $X$ be an arbitrary fibrant profinite 
$G$-spectrum. 
By letting $\{U_j\}_j$ be the collection of open normal subgroups of $G$, 
we can write $G=\lim_j G/U_j$. 
For each $j$, we let \[K_j:=KU_j,\] 
an open normal subgroup of $G$. For indices $i,j$ such that $U_j\subset U_i$, there is the canonical surjection $G/K_j\to G/K_i$, the natural $G/K_j$-equivariant map 
\[
X^{hK_i} = \mathrm{Map}(EG, X)^{K_i} \to \mathrm{Map}(EG, X)^{K_j} = X^{hK_j}
\] between fibrant spectra, and the isomorphism $G/K \cong \lim_j G/K_j$. Also, there is the induced map 
\[(X^{hK_i})^{h_dG/K_i} \overset{\simeq}{\longrightarrow} (X^{hK_j})^{h_dG/K_j}\] that 
is equal to the canonical composition
\[
(X^{hK_i})^{h_dG/K_i}
\to
F(E(G/{K_i})_+, X^{hK_i})^{G/{K_j}}
\to
F(E(G/{K_j})_+, X^{hK_j})^{G/{K_j}} 
\] 
and is a weak equivalence (for each $j$, since $G/K_j$ is a finite group, 
$(X^{hK_j})^{h_dG/K_j}$ is equivalent to $X^{hG}$, by \cite{hfplt}, as discussed earlier). Hence by taking the colimit over all $j$, 
we obtain an equivalence
\[
X^{hG} \simeq \colim_j (X^{hK_j})^{h_dG/K_j}.
\]
Hence, if $X^{hK}$ is a profinite $G/K$-spectrum and the colimit on the right-hand side above 
is equivalent to $(X^{hK})^{hG/K}$ (this is plausible if there is an equivalence of the 
right-hand side with $F(\lim_j E(G/K_j)_+, \colim_j X^{hK_j})^{G/K}$), then it would follow that (\ref{iteratedmap}) is an equivalence for the closed normal subgroup $K$. Unfortunately, this is not always the case.  
\begin{remark}\label{hyperfibrant}
For the 
duration of this remark (and the next), we let $K$ be an arbitrary closed subgroup of $G$. 
A problem similar to what was just described 
above occurs for discrete $G$-spectra: in \cite{davis2} a discrete $G$-spectrum $Y$ is called hyperfibrant if the canonical 
map $\colim_j Y^{h_dK_j} \to Y^{h_dK}$ is an equivalence for all $K$. In [loc.\,cit.], it was shown that if $Y$ is a hyperfibrant discrete $G$-spectrum and 
$H$ is any closed subgroup of $G$ that is normal in $K$, then the identification 
$Y^{h_dH} = \colim_j Y^{h_dHU_j}$ makes $Y^{h_dH}$ a discrete $K/H$-spectrum, 
and hence, $(Y^{h_dH})^{h_dK/H}$ is defined, and if $K$ is normal in $G$, then 
\[(Y^{h_dK})^{h_dG/K} \simeq Y^{h_dG}.\] 
\end{remark}

\begin{remark}
As in Remark \ref{hyperfibrant}, let $K$ be any closed subgroup of $G$. We recall how a version 
of the issue discussed above is handled in the corresponding ``pro-setting" in 
\cite[Section 11.1]{fausk}. If $Y$ is a pro-$G$-spectrum, in addition to considering 
the $K$-homotopy fixed point 
pro-spectrum $Y^{hK}$, Fausk uses a technique that was reviewed in Remark \ref{hyperfibrant}: he defines the $K$-$G$-homotopy fixed point pro-spectrum $Y^{h_GK}$ to 
be $\hocolim_j (Y_f)^{hK_j}$,
where $Y_f$ is a fibrant replacement of $Y$. 
Then Fausk shows that when $K$ is normal in $G$, 
there is an equivalence \[(Y^{h_GK})^{hG/K} \simeq Y^{hG}\] in the Postnikov 
model structure on pro-spectra. In \cite[Lemma 11.4]{fausk}, Fausk describes a situation 
when $Y^{h_GK}$ can be identified with $Y^{hK}$.
\end{remark}

Given the difficulties described above, our next step in studying the problem 
of iteration is to consider a special case. To help with this, we observe that because 
$X$ is a fibrant profinite $H$-spectrum for each closed subgroup $H$ of $G$, 
there is a canonical map
\[\colim_j X^{hK_j} \overset{\cong}{\longrightarrow} 
\colim_j \mathrm{Tot}(\mathrm{Map}(K_j^{\bullet}, X)) 
 \to \mathrm{Tot}(\mathrm{Map}(K^{\bullet}, X)) \overset{\cong}{\longrightarrow} 
 X^{hK},\] where here $K$ is any 
 closed subgroup of $G$ (for the isomorphisms, see \cite[Proposition 3.23]{hfplt}).
 A version of the following result 
in the setting of discrete $G$-spectra was obtained in \cite[end of \S 3]{davis2}.

\begin{lemma}\label{hyperfib}
Let $G$ be a profinite group and $K$ a closed 
subgroup. Let $X$ be an $f$-$G$-spectrum and $q$ an integer 
such that $\pi_tX=0$ for all $t\geq q$. Then the canonical map 
\[
\colim_j X^{hK_j} \overset{\simeq}{\longrightarrow} 
X^{hK}
\]
is an equivalence of spectra.
\end{lemma}
\begin{proof}
For each closed subgroup $H$ of $G$, by \cite[Proposition 3.23]{hfplt} 
the homotopy spectral sequence 
for $\mathrm{Tot}(\mathrm{Map}(H^\bullet, X))$ is a descent spectral sequence $E_r^{*,*}(H)$ that has the form
\[
E_2^{s,t}(H)=H^s_c(H; \pi_tX) \Rightarrow \pi_{t-s}(X^{hH}),
\] 
where the continuous cohomology group has coefficients in the finite discrete 
$H$-module $\pi_tX$.
For the subgroups $K$ and all $K_j$, the associated spectral sequences assemble to yield a map of conditionally convergent spectral sequences
\[
\colim_j E_r^{*,*}(K_j) \to E_r^{*,*}(K).
\]
Since $K=\bigcap_j K_j$, the $E_2$-terms satisfy
\[
\colim_j E_2^{s,t}(K_j) = \colim_j H^s_c(K_j; \pi_tX) \cong H^s_c(K; \pi_tX) = E_2^{s,t}(K).
\]
By \cite[Proposition 3.3]{mitchell}, the spectral sequence $\colim_j E_r^{*,*}(K_j)$ 
has abutment equal to the colimit of the abutments $\pi_\ast(X^{hK_j})$ if there exists a fixed $m$ such that $H_c^s(K_j; \pi_tX)=0$ for all $t\geq m$, all $s\geq 0$ and all $j$, and by hypothesis, 
this condition is satisfied. 
Hence, since the map of spectral sequences is an isomorphism 
from the $E_2$-terms onward, the map of 
abutments $\colim_j \pi_*(X^{hK_j}) \to \pi_*(X^{hK})$ is an isomorphism.
\end{proof}

\subsection{The notion of a $K$-Postnikov $G$-spectrum and its use with iterated 
continuous homotopy fixed points}\label{usedinintro} 

Let $X$ be an $f$-$G$-spectrum and let $K$ be a closed normal 
subgroup of $G$. 
As in \cite[Section 3.2]{hfplt}, for each integer $q$, let 
\[P^qX := \mathrm{cosk}_qX\] be the $q$th Postnikov section of $X$ in 
$\hSpg\mspace{1mu}\mathrm{:}$ since 
$X$ is an $f$-$G$-spectrum, each of $X$ 
and the Postnikov sections $P^qX$, for every $q$, are $f$-$H$-spectra, where 
$H$ is any closed subgroup of $G$. As explained in \cite[page~2888, end of \S 3]{davis2}, by 
applying 
the fact that each discrete $G$-spectrum $P^q X$ satisfies a coconnectivity condition, we see 
that 
\begin{itemize}
\item
for any closed subgroup $H$ in $G$, by \cite[Theorem 7.2]{delta}, there is an equivalence
\[(P^q X)^{h_dH} \simeq \holim_\Delta \mathrm{Map}_c(H^{\bullet+1}, P^qX)^H,\] 
and hence, Remark \ref{arbitrary} shows that
\[(P^q X)^{h_dH} \simeq (P^q X)^{hH},\] for every integer $q\mspace{1mu}\mathrm{;}$
\item
each $P^q X$ is a hyperfibrant discrete $G$-spectrum$\mspace{1mu}\mathrm{;}$ and thus,
\item 
there is the natural identification \[(P^q X)^{h_dK} = \colim_j (P^q X)^{h_dK_j},\] showing 
that each $(P^q X)^{h_dK}$ is a discrete $G/K$-spectrum.
\end{itemize}
The above discussion shows that  we can regard each $G/K$-spectrum $(P^qX)^{hK}$ as 
a discrete $G/K$-spectrum by the identification \[(P^q X)^{hK} = (P^q X)^{h_dK}\] 
(this 
identification can also be obtained by using Lemma \ref{hyperfib}). 

Before giving our next result, we introduce some helpful 
terminology. We note that though the following definition is not short, 
it is also not complicated: the crux of the notion it defines 
(``$K$-Postnikov $G$-spectrum") is almost completely captured by its property (a), and each of properties (b), (c) and (d) is just a basic compatibility condition 
that one would expect to be satisfied when the zigzags in the defining data satisfy a 
minimal level of naturality.

\begin{defn}\label{postnikovK}
Let $X$ be an $f$-$G$-spectrum and let $K$ be a closed normal subgroup of $G$. 
Then we say that $X$ is a $K$-Postnikov $G$-spectrum if there is 
an inverse system $\{X^q(K)\}_{q \in \mathbb{Z}}$ of $G/K$-spectra that has 
the following properties$\mspace{1mu}\mathrm{:}$
\begin{itemize}
\item[(a)]
for each $q$, the spectrum $X^q(K)$ is an $f$-$G/K$-spectrum and a 
model for the $G/K$-spectrum 
$(P^q X)^{hK}$ in the category of profinite $G/K$-spectra$\mspace{1mu}\mathrm{;}$ 
\item[(b)]
for each $q$, there is an equivalence $\bigl((P^qX)^{h_dK}\bigr)^{h_dG/K} 
\simeq (X^q(K))^{h_dG/K}\mathrm{;}$ 
\item[(c)]
there is an equivalence
\[\holim_q \bigl((P^qX)^{h_dK}\bigr)^{h_dG/K} \simeq \holim_q (X^q(K))^{h_dG/K};\] 
and
\item[(d)] 
the fibrant profinite $G/K$-spectrum
$\holim_q X^q(K)$ is a model for the 
$G/K$-spectrum $\holim_q (P^qX)^{hK}$ 
in the category of profinite $G/K$-spectra.
\end{itemize}  
\end{defn}

\begin{remark} 
If $X$ is a $K$-Postnikov $G$-spectrum, then because 
each $X^q(K)$ is a 
discrete $G/K$-spectrum, it follows automatically that each map
$X^q(K) \to X^{q-1}(K)$ has the requisite continuity
properties for being a morphism in the category of profinite $G/K$-spectra, and hence, 
the inverse system 
$\{X^q(K)\}_{q \in \mathbb{Z}}$ 
is a diagram in that category.
\end{remark}

Our next result illustrates the utility of Definition \ref{postnikovK}. After proving 
this result, we give a discussion of when an $f$-$G$-spectrum possesses the 
properties required by this definition.

Several key ingredients in the next result and its proof are from an 
unpublished manuscript by the first author (however, the proof below does not 
depend in any way on this manuscript)
that uses Postnikov towers to study 
the problem of iterated homotopy fixed points in the setting of discrete $G$-spectra. 
The aforementioned ingredients and manuscript build on the idea that Postnikov 
towers are a helpful tool for building homotopy fixed point spectra for profinite 
group actions, and as 
far as we know, this idea 
is primarily due to \cite{fausk} (and \cite{goerss, jardine}). 

\begin{theorem}\label{fiterated}
Let $G$ be a profinite group  
and let $K$ be a closed normal subgroup of $G$ such that $G/K$ has finite virtual cohomological dimension. If $X$ is a $K$-Postnikov $G$-spectrum, 
then there is an equivalence 
\[
X^{hG} \simeq (X^{hK})^{hG/K}.
\] 
\end{theorem}
\begin{proof}
There is the following chain of equivalences:
\begin{align*}
X^{hG} & = \mathrm{Map}_G(EG, \lim_q P^q X) 
\simeq \holim_{q} (P^q X)^{hG} \simeq \holim_q (P^q X)^{h_dG} 
\\ & \simeq \holim_q \bigl(\colim_j (P^q X)^{h_d K_j}\bigr)^{h_dG/K}
\simeq \holim_q (X^q(K))^{h_dG/K} \\ & \simeq \holim_q (X^q(K))^{hG/K} 
\cong \mathrm{Map}_{G/K}(E(G/K), \holim_q X^q(K)) \\
& =  \bigl(\holim_q X^q(K)\bigr)^{hG/K}
= (X^{hK})^{hG/K}.
\end{align*} 
Above, the second equivalence is from an application of the fact that each $P^q X$ is a fibrant profinite 
$G$-spectrum and, as in \cite[page~205, (9)]{hfplt}, $\{\mathrm{Map}_G(EG, 
{P^qX})\}_{q \in \mathbb{Z}}$ 
is a tower of fibrations of fibrant objects in spectra; the fourth equivalence 
follows from \cite[proof of Lemma 4.9]{davis2};  
the sixth equivalence is due to Theorem \ref{profdisc}; 
the fibrancy 
of $\holim_q X^q(K)$ as a profinite $G/K$-spectrum implies the next-to-last 
equivalence (an equality); 
and the last equivalence, an identity that is an example of 
Definition \ref{itdef}, is because the equivalences 
\[\holim_q X^q(K) \simeq \holim_q (P^qX)^{hK} 
\simeq \lim_q (P^q X)^{hK} \cong \mathrm{Map}_K(EG, \lim_q P^q X) 
= X^{hK}\] show that $\holim_q X^q(K)$ is a model for $X^{hK}$ in the category 
of profinite $G/K$-spectra.
\end{proof}

\begin{remark}\label{zyzzyva}
In the proof of Theorem \ref{fiterated}, we showed that 
%the $G/K$-spectrum 
$\holim_q X^q(K)$ is a profinite $G/K$-model for $X^{hK}$. 
Now suppose that $Y$ is another profinite $G/K$-model for $X^{hK}$. 
For this abstract situation, as explained in Remark \ref{zigzags}, we do not know in general that there must be an equivalence $(\holim_q X^q(K))^{hG/K} \simeq 
Y^{hG/K}$. However, we do not regard this lack of an equivalence as problematic: we believe that in most specific situations, $\holim_q X^q(K)$ and $Y$ will be related in a concrete way that allows one to obtain an equivalence between their $G/K$-homotopy fixed point spectra. 
For example, in Remark \ref{compatible}, two models ($F^s_G X$ and $F^s_HX$) are related in a specific way that gives equivalent 
homotopy fixed point spectra, 
by an argument that is, in an appreciable way, simpler than the abstract 
argument of Remark \ref{zigzags} (which also gives the desired equivalence).
\end{remark}

\begin{remark} 
The first displayed line of the proof of Theorem \ref{fiterated} shows that if $G$ is an arbitrary profinite 
group and $X$ is an $f$-$G$-spectrum, then there are equivalences 
\begin{equation}\label{threeterms}
X^{hG} \simeq \holim_q (P^qX)^{hG} \simeq \holim_q (P^q X)^{h_dG}.
\end{equation} 
The equivalence between the first and third expressions in (\ref{threeterms}) 
gives 
\[X^{hG} \simeq \holim_{q \geq 0} (P^q X)^{h_dG} 
= \Bigl(\mspace{-1.2mu}\holim_{q \geq 0} P^q X\mspace{-1.2mu}\Bigr)^{
\negthinspace \mspace{-.5mu}hG},\] where the last expression, 
$\bigl(\holim_{q \geq 0} P^q X\bigr)^{\negthinspace \mspace{.5mu} hG}$, 
is the $G$-homotopy fixed point spectrum of the continuous $G$-spectrum 
$\holim_{q \geq 0} P^q X$ (in the sense of \cite{davis}; $\{P^q X\}_{q \geq 0}$ 
is a tower of discrete $G$-spectra that are fibrant as spectra). This observation
gives a version of Theorem \ref{profdisc} with the hypothesis on 
cohomological dimension removed.
\end{remark}

\begin{remark}\label{given_later}
Let $G$ be a strongly complete profinite group. Just after Definition \ref{twodefs} (see (\ref{itturnsoutyes!})), we stated that if 
$h \colon X \to Y$ is a $G$-equivariant map between 
$\pi$-finite $G$-spectra such that $h$ is a stable equivalence 
of spectra, then the map
\[(F^s_G(h))^{hG} \colon X^{hG} = (F^s_GX)^{hG} \xrightarrow{\,\simeq\,} 
(F^s_GY)^{hG} = Y^{hG}\] is an equivalence. Now we are able to 
give a proof of this assertion. First, as noted just after Theorem 
\ref{Gstablefinitecompletion}, 
the map $F^s_G(h)$ is a stable 
equivalence of spectra. Also, 
if $Z$ is a fibrant profinite $G$-spectrum, recall 
that its profinite homotopy groups $\pi_\ast(Z)$ are isomorphic to the 
stable homotopy groups $\pi_\ast(UZ)$ of the fibrant spectrum $UZ$ that underlies $Z$ (see \cite[Proposition 4.8, (b)]{gspectra}). Then for all integers $q$ and $t$, since $P^q(F^s_GX)$ and $P^q(F^s_GY)$ are fibrant profinite $G$-spectra,   
\[\pi_t(UP^q(F^s_GX)) \cong \pi_t(P^q(F^s_GX)) \cong \begin{cases}
\pi_t(F^s_GX) \cong \pi_t(UF^s_GX) \cong \pi_t(UF^s_GY), & t \leq q; \\
0, & t > q,\end{cases}\] so that each induced map $P^q(F^s_GX) 
\to P^q(F^s_GY)$ is a weak equivalence of discrete $G$-spectra 
(because each such map 
induces an isomorphism on stable homotopy groups, and 
hence, is an equivalence of spectra). 
It follows 
that for every $q$, the map \[(P^q(F^s_G(h)))^{h_dG} \colon (P^q(F^s_GX))^{h_dG} \xrightarrow{\,\simeq\,} 
(P^q(F^s_GY))^{h_dG}\] 
is an equivalence between 
% fibrant [[ when this line is uncommented, it gives a true statement but
%			saying ``fibrant" wasn't needed, so I omitted it. ]]
spectra. 
Also, there is the commutative diagram 
\[\xymatrix{
(F^s_GX)^{hG} \ar[r]^-{(F^s_G(h))^{hG}} & (F^s_GY)^{hG}\\
\mathrm{Map}_G(EG, \lim_q P^q(F^s_GX)) \ar[u]^-\simeq \ar[r] 
\ar[d]_-{\simeq} 
& \mathrm{Map}_G(EG, \lim_q P^q(F^s_GY)) \ar[u]_-\simeq 
\ar[d]^-{\simeq}\\ 
\holim_q \mathrm{Map}_G(EG, P^q(F^s_GX)) 
\ar[r] & \holim_q \mathrm{Map}_G(EG, P^q(F^s_GY))
%\holim_q (P^q(F^s_GX))^{h_dG} \ar[r]^-\simeq &
%\holim_q (P^q(F^s_GY))^{h_dG}
}\] (as noted in (\ref{identification}), 
the two upward-pointing vertical maps are equivalences and, 
as in the proof of Theorem \ref{fiterated}, the two downward-pointing 
vertical maps are equivalences), 
which reduces showing that $(F^s_G(h))^{hG}$ is 
an equivalence to showing that the 
bottommost horizontal map is an equivalence, and 
therefore, we are done if we can show that for every $q$, the map 
\[\mu^q := \mathrm{Map}_G(EG, P^q(F^s_G(h))) \colon 
\mathrm{Map}_G(EG, P^q(F^s_GX)) 
\to \mathrm{Map}_G(EG, P^q(F^s_GY))\] is 
an equivalence (the source and target of this map 
are fibrant spectra, by \cite[Lemma 3.10]{hfplt}). Each map $\mu^q$ 
can be identified with the map 
\[(P^q(F^s_G(h)))^{hG} = \mathrm{Map}_G(EG, R_G(P^q(F^s_G(h)))) \colon 
(P^q(F^s_GX))^{hG} \to (P^q(F^s_GY))^{hG},\] and this last map is 
an equivalence because 
\begin{equation}\label{tofinishthisremark}
(P^q(F^s_GX))^{hG} \simeq (P^q(F^s_GX))^{h_dG} 
\xrightarrow{\,\simeq\,} (P^q(F^s_GY))^{h_dG} 
\simeq (P^q(F^s_GY))^{hG},
\end{equation} 
where the first and last equivalences in (\ref{tofinishthisremark}) follow 
from the first bullet point in the discussion at the beginning of Section \ref{usedinintro}. This 
completes the proof of the desired result.
\end{remark} 

As in Theorem \ref{fiterated}, we continue to let 
$K$ be a closed normal subgroup of $G$ and we let $X$ be 
an $f$-$G$-spectrum. Theorem \ref{fiterated} shows that 
it is quite useful when $X$ is a $K$-Postnikov $G$-spectrum and so 
we now give a discussion about when this occurs.

Suppose that $(P^qX)^{hK}$ is a $\pi$-finite $G/K$-spectrum, for each 
integer $q$, with 
$G/K$ strongly complete. (It is worth noting that if 
$G$ is strongly complete, then so is $G/K$.)
By Lemma \ref{hyperfib}, the $G/K$-equivariant map
\[(P^q X)^{hK} = \mathrm{Map}(EG, P^qX)^K \overset{\simeq}\longleftarrow \colim_j \mathrm{Map}(EG, P^qX)^{K_j} 
= \colim_j (P^q X)^{hK_j}\] is a weak equivalence of spectra. The continuous surjection 
$G/K \to G/K_j$ makes the discrete $G/K_j$-spectrum $\mathrm{Map}(EG, P^qX)^{K_j}$ 
a discrete $G/K$-spectrum, so that the $\pi$-finite $G/K$-spectrum 
$\colim_j \mathrm{Map}(EG, P^qX)^{K_j}$ is a 
discrete $G/K$-spectrum. Notice that the inverse system
\begin{equation}\label{system}
\bigl\{F^s_{G/K}\bigl(\colim_j \mathrm{Map}(EG, P^qX)^{K_j}\bigr)\bigr\}_{q \in \mathbb{Z}}\end{equation} is 
a diagram of profinite $G/K$-spectra, each of which is an $f$-$G/K$-spectrum and a 
profinite $G/K$-model for $(P^qX)^{hK}$. Since $(P^qX)^{hH} \simeq (P^qX)^{h_dH}$ 
for all closed subgroups $H$, the discrete $G/K$-spectrum 
$\colim_j \mathrm{Map}(EG, P^qX)^{K_j}$ can be identified with the 
discrete $G/K$-spectrum $(P^qX)^{h_dK} = \colim_j (P^qX)^{h_dK_j}$. Because 
\[\colim_j \mathrm{Map}(EG, P^qX)^{K_j} \overset{\simeq}{\longrightarrow} 
F^s_{G/K}\bigl(\colim_j \mathrm{Map}(EG, P^qX)^{K_j}\bigr)\] is a weak equivalence 
of discrete $G/K$-spectra, there is a weak equivalence
\[\bigl(\colim_j \mathrm{Map}(EG, P^qX)^{K_j}\bigr)^{h_dG/K} 
\overset{\simeq}{\longrightarrow} 
\bigl(F^s_{G/K}\bigl(\colim_j \mathrm{Map}(EG, P^qX)^{K_j}\bigr)\bigr)^{h_dG/K}.\] The 
preceding two sentences show that the system in (\ref{system}) satisfies 
property (b) of Definition \ref{postnikovK}, and additionally, it is easy to 
see that properties (c) and (d) are satisfied, completing a proof of the following 
result.

\begin{theorem}\label{postnikovresult}
Let $K$ be a closed normal subgroup of the profinite group $G$ and 
let $X$ be an $f$-$G$-spectrum. If $G/K$ is strongly complete 
and, for every $q \in \mathbb{Z}$, 
$(P^q X)^{hK}$ has finite homotopy groups,
then $X$ is a $K$-Postnikov $G$-spectrum.
\end{theorem}

Notice that in Theorem \ref{postnikovresult}, the hypotheses of the first 
sentence imply that 
for each $q \in \mathbb{Z}$, since $P^qX$ is an $f$-$K$-spectrum, $\pi_t(P^qX)$ is a finite discrete $K$-module 
for every integer $t$. 
\begin{cor}\label{goodtool}
Let $G$, $K$, and $X$ be as in the first sentence of Theorem $\mathrm{\ref{postnikovresult}}$ and 
suppose that $G/K$ is strongly complete. If $H^s_c(K; \pi_t(X))$ is a finite group 
for all $s \geq 0$ and every integer $t$, then $X$ is a $K$-Postnikov $G$-spectrum.
\end{cor}
\begin{proof}
By Theorem \ref{postnikovresult}, it suffices to show that for each integer $q$, 
the spectrum $(P^q X)^{hK}$ has finite homotopy groups. We consider the 
conditionally convergent descent spectral sequence
\[E_2^{s,t} = H^s_c(K; \pi_t(P^qX)) \Rightarrow \pi_{t-s}((P^qX)^{hK}).\] 
For all $s \geq 0$ and any $t > q$, the vanishing of $\pi_t(P^qX)$ implies 
that $E_2^{s,t} = 0$. It follows that the filtration on the abutment is finite 
(for example, see \cite[Lemma 5.48]{thomason}). Since any nonzero 
terms on the $E_2$-page of the above spectral sequence 
are finite, we obtain the desired conclusion.
\end{proof}

\begin{remark}\label{symondsremark}
In Corollary \ref{goodtool}, the hypotheses imply that 
$\pi_t(X)$ is a finite discrete $K$-module for 
every integer $t$. Then it is worth noting that in Corollary \ref{goodtool}, the 
finiteness condition on the continuous cohomology groups is plausible: 
for example, by \cite[Proposition 4.2.2]{symondsweigel}, if $K$ is of 
type $p$-$\mathbf{FP}_\infty$ and $M$ is a finite discrete ($p$-torsion) $\mathbb{Z}_p[[K]]$-module, then $H^s_c(K;M)$ is finite, for all $s \geq 0$ (we refer the reader to 
\cite{symondsweigel} for more details about this result). 
\end{remark}

Let $q$ be any fixed integer. We consider in more detail the condition that 
$(P^qX)^{hK}$ has finite homotopy groups. As above, we assume 
that $K$ is closed and normal in $G$ and $X$ is an $f$-$G$-spectrum.
  
Suppose that $K$ contains an open normal subgroup 
$U_{\negthinspace \scriptscriptstyle{K}}$ such that 
the abelian groups $H^s_c(U_{\negthinspace \scriptscriptstyle{K}}; \pi_t(P^qX))$ are 
finite for all $s \geq 0$ and 
all $t \leq q$. As 
recalled earlier (from \cite{hfplt}), 
since $P^q X$ is a fibrant profinite $K$-spectrum, there is a 
weak equivalence
\[
(P^q X)^{hK} \overset{\simeq}{\longrightarrow} 
\bigl((P^q X)^{hU_{\negthinspace \scriptscriptstyle{K}}}\bigr)^{\negthinspace \mspace{1mu} h_dK/{U_{\negthinspace \scriptscriptstyle{K}}}}.
\]
By arguing as in the proof of Corollary \ref{goodtool}, the 
descent 
spectral sequence
\[H^s_c(U_{\negthinspace \scriptscriptstyle{K}}; \pi_t(P^q X)) \Rightarrow 
\pi_{t-s}((P^q X)^{hU_{\negthinspace \scriptscriptstyle{K}}})\] 
yields that $(P^q X)^{hU_{\negthinspace \scriptscriptstyle{K}}}$
is a $\pi$-finite $K/{U_{\negthinspace \scriptscriptstyle{K}}}$-spectrum. 

\begin{remark}\label{wepausehere}
Without making any additional assumptions, we pause 
to consider the $\pi$-finite $K/{U_{\negthinspace \scriptscriptstyle{K}}}$-spectrum $(P^q X)^{hU_{\negthinspace \scriptscriptstyle{K}}}$. 
Since the finite group 
$K/{U_{\negthinspace \scriptscriptstyle{K}}}$ is 
naturally discrete, it is strongly complete, and hence, by 
Theorem \ref{Gstablefinitecompletion}, there is a $K/{U_{\negthinspace \scriptscriptstyle{K}}}$-equivariant map
\[
(P^q X)^{hU_{\negthinspace \scriptscriptstyle{K}}} \overset{\simeq}{\longrightarrow} 
F^s_{K/{U_{\negthinspace \scriptscriptstyle{K}}}}((P^q X)^{hU_{\negthinspace \scriptscriptstyle{K}}})
\]
that 
is a weak equivalence, with target equal to an $f$-$K/{U_{\negthinspace \scriptscriptstyle{K}}}$-spectrum, yielding a particularly nice model for $(P^q X)^{hU_{\negthinspace \scriptscriptstyle{K}}}$. 
It follows that there are equivalences of spectra
\[(P^q X)^{hK} \simeq \bigl((P^q X)^{hU_{\negthinspace \scriptscriptstyle{K}}}\bigr)^{\negthinspace \mspace{1mu} h_dK/{U_{\negthinspace \scriptscriptstyle{K}}}} 
\simeq \holim_{K/{U_{\negthinspace \scriptscriptstyle{K}}}} F^s_{K/{U_{\negthinspace \scriptscriptstyle{K}}}}((P^q X)^{hU_{\negthinspace \scriptscriptstyle{K}}}).\] The last spectrum 
above can be regarded as a homotopy limit of a fibrant 
profinite spectrum in the category of profinite spectra, showing 
that $(P^q X)^{hK}$ can be regarded as a profinite spectrum. 
\end{remark}

Now we go a little further with the above conclusion that 
$(P^q X)^{hU_{\negthinspace \scriptscriptstyle{K}}}$ is a 
$\pi$-finite $K/{U_{\negthinspace \scriptscriptstyle{K}}}$-spectrum. 
Since $K/U_{\negthinspace \scriptscriptstyle{K}}$ is a finite group and 
$\pi_t((P^q X)^{hU_{\negthinspace \scriptscriptstyle{K}}})$ is finite for every integer $t$, the groups $H^s(K/{U_{\negthinspace \scriptscriptstyle{K}}}; 
\pi_t((P^q X)^{hU_{\negthinspace \scriptscriptstyle{K}}}))$ are finite for all $s$ and $t$. Therefore, if for all integers $t$, $H^s(K/{U_{\negthinspace \scriptscriptstyle{K}}}; \pi_t((P^qX)^{hU_{\negthinspace \scriptscriptstyle{K}}}))$ is zero for all $s \geq r$, for some positive integer $r$, 
then once again, 
the descent spectral sequence
\[H^s(K/{U_{\negthinspace \scriptscriptstyle{K}}}; \pi_t((P^qX)^{hU_{\negthinspace \scriptscriptstyle{K}}})) \Rightarrow \pi_{t-s}((P^qX)^{hK})\] yields 
that $(P^qX)^{hK}$ has finite homotopy groups, as desired.

\subsection{More cases of 
well-behaved iteration via applications of Theorem \ref{fiterated}} 
To obtain these additional cases, we begin with the following definition, a natural extension of 
Definition \ref{twodefs}. 

\begin{defn}\label{extend} 
If $G$ is a strongly complete profinite group and $X$ is a $\pi$-finite $G$-spectrum, then 
$F^s_GX$ is both 
a fibrant profinite $H$-spectrum and a discrete $H$-spectrum for any closed subgroup $H$ in 
$G$, and thus, for such $H$, we define
\[X^{hH} := (F^s_GX)^{hH}, \ \ \ X^{h_dH} := (F^s_GX)^{h_dH}.\] 
\end{defn}

In the above definition, when the closed subgroup 
$H$ is a proper subgroup of $G$ that is 
strongly complete, then 
since $X$ is a $\pi$-finite $H$-spectrum, the definition also yields that 
$X^{hH} := (F^s_HX)^{hH}$ and $X^{h_dH} := (F^s_H X)^{h_dH}$. Thus, in 
the remark 
below, we show that our two definitions of $X^{hH}$ agree with each other and 
that our two definitions of $X^{h_dH}$ are equivalent.

\begin{remark}\label{compatible}
Let $G$ and $X$ be as in Definition \ref{extend} and let $H$ be a strongly 
complete closed subgroup of $G$. Then there 
is a commutative diagram 
\[\xymatrix{
X \ar[r]^-\simeq  \ar[d]_-\simeq & F^s_HX \ar[d]^-\simeq  
\\ 
F^s_GX \ar[r]^-\simeq & F^s_H(F^s_GX) &
}\] of weak equivalences in spectra that 
are $H$-equivariant. It is easy to see that 
\[F^s_GX \overset{\simeq}{\longrightarrow} F^s_H(F^s_GX) \overset{\simeq}{\longleftarrow} 
F^s_HX\] is a zigzag of weak equivalences in the category of discrete $H$-spectra, 
and hence, \[(F^s_GX)^{h_dH} \simeq (F^s_HX)^{h_dH},\] as desired. Now, as in 
\cite[Section 3.2]{hfplt}, for each integer $q$, 
we let $P^q$ denote the functor \[\mathrm{cosk}_q \colon \mathrm{Sp}(\hat{\mathcal{S}}_{\ast H}) 
\to \mathrm{Sp}(\hat{\mathcal{S}}_{\ast H}), \ \ \ Z \mapsto P^qZ := \mathrm{cosk}_qZ.\] 
Notice that 
\[(F^s_GX)^{hH} \simeq \holim_q (P^q(F^s_GX))^{hH} \simeq \holim_q (P^q (F^s_GX))^{h_dH},\] 
as in the 
proof of Theorem \ref{fiterated}. Similarly, we have
\[(F^s_HX)^{hH} \simeq \holim_q (P^q(F^s_HX))^{h_dH}.\] Also, by considering 
homotopy groups, 
it is easy to see that for each integer $q$,
\[P^q(F^s_GX) \overset{\simeq}{\longrightarrow} P^q(F^s_H(F^s_GX)) \overset{\simeq}{\longleftarrow} 
P^q(F^s_HX)\] is a zigzag of weak equivalences in the category of discrete $H$-spectra. 
It follows that 
\[(P^q(F^s_GX))^{h_dH} \overset{\simeq}{\longrightarrow} (P^q(F^s_H(F^s_GX)))^{h_dH} \overset{\simeq}{\longleftarrow} 
(P^q(F^s_HX))^{h_dH}\] is a zigzag of weak equivalences between fibrant spectra 
that is natural in $q$, and hence, 
\begin{equation}\label{forbelowonly} 
(F^s_GX)^{hH} \simeq (F^s_HX)^{hH}.
\end{equation}
Since $F^s_GX$ is an $f$-$H$-spectrum, the equivalence in (\ref{forbelowonly}) can 
also be obtained by applying the argument in Remark \ref{zigzags} to 
the zigzag $F^s_GX \xleftarrow{\,\simeq\,} X \xrightarrow{\,\simeq\,} F^s_H X$. 
However, the argument given in the present remark is preferable to that of 
Remark \ref{zigzags} in the sense that the crux of the former argument is a zigzag of length two, whereas the latter argument requires a zigzag of length four.
\end{remark}

\begin{remark}\label{moregeneral}
In Definition \ref{extend}, if $H = G$, we already know that the 
two notions of homotopy fixed points are homotopy invariant. In fact, 
as we now explain, these two notions are homotopy invariant for 
all $H$ (even if $H$ is not strongly complete). Let $G$ be strongly 
complete, let $H$ be any closed subgroup of $G$, and let $h \colon 
X \to Y$ be a stable equivalence of spectra that is $G$-equivariant, 
with $X$ and $Y$ $\pi$-finite $G$-spectra. Since $F^s_G(h)$ is a 
weak equivalence of discrete $G$-spectra, it is a weak equivalence of 
discrete $H$-spectra, and hence, the map
\[(F^s_G(h))^{h_dH} \colon X^{h_dH} = (F^s_GX)^{h_dH} \xrightarrow{\,\simeq\,}
(F^s_GY)^{h_dH} = Y^{h_dH}\] is an equivalence. Now suppose that 
$W$ and $Z$ %are $f$-$G$-spectra, 
are fibrant profinite $G$-spectra that are $\pi$-finite as 
$G$-spectra (these hypotheses 
are satisfied for example, if $W$ and $Z$ are $f$-$G$-spectra), 
with $\ell \colon W \to Z$ 
a map of profinite $G$-spectra 
%(notice that these hypotheses imply that 
%$\ell$ is automatically a morphism of profinite $G$-spectra) 
that is a stable equivalence of spectra. 
Then $W$ and $Z$ are fibrant profinite $H$-spectra, $\ell$ is a map of profinite $H$-spectra, and (a reference for the content in the remainder of this sentence is 
the material in \cite[\S 3.4 (after the proof of Lemma 3.21)]{hfplt}) there is the induced morphism 
\[\mathrm{Map}(H^\bullet, \ell) \colon 
\mathrm{Map}(H^\bullet, W) \to \mathrm{Map}(H^\bullet, Z)\] of cosimplicial 
spectra, which induces a morphism from descent spectral sequence 
\[E_2^{s,t} = H^s_c(H; \pi_tW) \Rightarrow \pi_{t-s}(W^{hH})\] to 
descent spectral sequence
\[E_2^{s,t} = H^s_c(H; \pi_tZ) \Rightarrow \pi_{t-s}(Z^{hH}).\] 
For any $t \in \mathbb{Z}$, the induced map $\pi_t(\ell) \colon 
\pi_tW \to \pi_tZ$ between 
the stable profinite homotopy groups is an 
isomorphism between finite discrete $H$-modules (for example, the 
stable profinite homotopy group $\mathcal{P}_t := \pi_t(R_GW)$ is Hausdorff and isomorphic to the 
finite stable homotopy group $\pi_t(W)$, so that the topology of the profinite 
space $\mathcal{P}_t$ is the discrete topology), and hence, the 
$E_2$-terms of the two conditionally convergent spectral sequences above 
are isomorphic, yielding successively that the map $\pi_\ast(\ell^{hH})$ is an 
isomorphism and, finally, $\ell^{hH} \colon W^{hH} \to 
Z^{hH}$ is a stable equivalence. This conclusion implies that 
\[(F^s_G(h))^{hH} \colon X^{hH} = (F^s_GX)^{hH} 
\xrightarrow{\,\simeq\,}
(F^s_GY)^{hH} = Y^{hH}\] is an equivalence, as desired. With this 
argument and Remark \ref{given_later}, we now have two proofs that 
given $h$ as above, $(F^s_G(h))^{hG}$ is an equivalence: though the 
argument in this remark is shorter, we do not regard this situation as 
one of excess, since the observations and main techniques (utilizing Postnikov 
towers in the world of profinite $G$-spectra and homotopy fixed points for discrete 
$G$-spectra) in Remark \ref{given_later}'s 
proof are interesting and useful. 
\end{remark}

Now that we have shown that Definition \ref{extend} is robust, we can 
put it to work with the following result.

\begin{theorem}\label{pifiniteiterated} 
Let $G$ be a strongly complete profinite group with 
$K$ a closed normal subgroup of $G$ such that $G/K$ has finite virtual cohomological dimension. 
If $X$ is a $\pi$-finite $G$-spectrum such that $F^s_GX$ is a $K$-Postnikov $G$-spectrum, 
then there is an equivalence of spectra
\[ 
X^{hG} \simeq (X^{hK})^{hG/K}.
\]
\end{theorem}
\begin{proof}
By Theorem \ref{fiterated}, we have
\[X^{hG} = (F^s_G X)^{hG} \simeq \bigl((F^s_GX)^{hK}\bigr)^{\negthinspace \mspace{1mu} hG/K} 
= (X^{hK})^{hG/K}. 
\qedhere 
\]
\end{proof}

The following result (which is actually just an assemblage of well-known facts) gives an example of a single hypothesis that 
is sufficient to yield the group-theoretic conditions that are required 
by Theorem \ref{pifiniteiterated} (in its first sentence).

\begin{lemma}\label{groupresult}
Suppose that 
$\mathbb{G}$ is a profinite group that is $p$-adic analytic. 
If $G$ is a closed subgroup of $\mathbb{G}$ and $K$ is a 
closed normal subgroup of $G$, then $G$ and $G/K$ 
are strongly complete and $G/K$ has finite virtual cohomological 
dimension. 
\end{lemma}

\begin{proof}
Since $\mathbb{G}$ is 
$p$-adic analytic, \cite[Theorem 9.6]{dsms} implies that $G$ and 
the quotient group $G/K$ are $p$-adic analytic groups. 
Since $G$ is $p$-adic 
analytic, it contains an open pro-$p$ subgroup $H$ that is finitely generated. It follows 
that $H$ is strongly complete, 
%If $G$ is not strongly complete, then it contains no open 
% subgroups that are strongly complete, 
and hence, $G$ and $G/K$ are strongly complete. Also, since $G/K$ is $p$-adic analytic, it 
has finite virtual cohomological 
dimension (for example, see \cite[page 330]{davis}). 
\end{proof}

We point out that besides implying the group-theoretic conditions of Theorem \ref{pifiniteiterated}, the single hypothesis of Lemma 
\ref{groupresult} also implies (as stated explicitly in the lemma) that 
$G/K$ is strongly complete, and this property is required in 
Theorem \ref{postnikovresult} and Corollary \ref{goodtool}.

Now we want to apply Theorem \ref{fiterated} to suitable homotopy 
limits of profinite $G$-spectra: to do this, we need the following definition, whose 
key components are properties (a) and (b) below.
\begin{defn}\label{KGbeta}
Let $K$ be a closed normal subgroup of $G$ and let $J$ be a small category. 
We say that $\{X_\beta\}_{\beta \in J}$  
is a diagram of $K$-Postnikov $G$-spectra that is natural in $\beta$ whenever 
\begin{itemize}
\item[(a)]
$\{X_\beta\}_{\beta \in J}$ is a $J$-shaped diagram of 
profinite $G$-spectra, such that 
\item[(b)]
for each $\beta$, $X_\beta$ is a $K$-Postnikov 
$G$-spectrum, with $\{X^q_\beta(K)\}_{q \in \mathbb{Z}}$ denoting
the inverse system of profinite $G/K$-models associated to 
$X_\beta$, 
\item[(c)]
there is a $(J \times \{q\})$-shaped 
diagram $\{X^q_\beta(K)\}_{\beta, q}$ of profinite $G/K$-spectra, 
\item[(d)]
there is an equivalence
\[\holim_\beta \holim_q\bigl((P^qX_\beta)^{h_dK}\bigr)^{\negthinspace \mspace{1mu} h_dG/K}
\simeq \holim_\beta \holim_q (X^q_\beta(K))^{h_dG/K},\] and 
\item[(e)]
the fibrant profinite $G/K$-spectrum $\holim_\beta \holim_q X^q_\beta(K)$ is a model for the 
$G/K$-spectrum $\holim_\beta \holim_q (P^q X_\beta)^{hK}$ in the category of 
profinite $G/K$-spectra.
\end{itemize} 
\end{defn}
\begin{remark}\label{aboutprevious}
In the above definition, we point out that 
properties (c), (d) and (e) are 
criteria that one would expect to be satisfied on the basis of just (a) and 
(b) alone: they are just expressing the requirement that the input data have 
enough naturality to be practically useful (e.g., properties (d) and (e) 
are saying that properties (c) and (d), respectively, in Definition 
\ref{postnikovK} are natural in $\beta$).  
\end{remark}

Notice that if $\{X_\beta\}_{\beta \in J}$ 
is a diagram of $K$-Postnikov $G$-spectra that is 
natural in $\beta$, then $\holim_\beta X_\beta$ is a fibrant profinite $G$-spectrum. 

\begin{theorem}\label{iteratedthm}
Let $G$ be a profinite group and K a closed normal subgroup of $G$ such that $G/K$ has finite virtual cohomological dimension. 
If $\{X_\beta\}_{\beta \in J}$ is a 
diagram of $K$-Postnikov $G$-spectra that is natural in $\beta$, then there is an 
equivalence of spectra
\[\bigl(\holim_\beta X_\beta\bigr)^{\negthinspace \mspace{1mu} hG} \simeq 
\Bigl(\mspace{-2mu}\bigl(\holim_\beta X_\beta\bigr)^{\negthinspace 
\mspace{1mu} hK}\Bigr)^{\negthinspace \mspace{-.5mu} hG/K}\]
and there is a conditionally convergent 
spectral sequence  
$$E_2^{s,t}=H_c^s\bigl(G/K;\pi_t\bigl((\holim_\beta X_\beta)^{hK}\bigr)\bigr) \Rightarrow \pi_{t-s}\bigl((\holim_\beta X_\beta)^{hG}\bigr),$$ 
where the $E_2$-term is the continuous cohomology of $G/K$ with coefficients the profinite $G/K$-module $\pi_t\bigl((\holim_\beta X_\beta)^{hK}\bigr)$. 
\end{theorem}
\begin{proof}
The $f$-$G$-spectra $X_{\beta}$ satisfy the hypotheses of Theorem \ref{fiterated}. Hence for each $\beta$, we have an equivalence  
\[
(X_{\beta})^{hG} \simeq ((X_{\beta})^{hK})^{hG/K} = \bigl(\holim_q X^q_\beta (K)\bigr)^{\negthinspace \mspace{1mu} hG/K}.
\] 
Since taking $H$-homotopy fixed points commutes with homotopy limits 
of fibrant profinite $H$-spectra, for any profinite group $H$, by \cite[Proposition 3.12]{hfplt}, and the $G/K$-spectrum $(\holim_\beta X_\beta)^{hK}$ is 
easily seen to have $\holim_\beta \holim_q X^q_\beta(K)$ as a profinite $G/K$-model, 
it follows that there is an equivalence
\[
(\holim_\beta X_\beta)^{hG} \simeq \bigl(\holim_\beta \holim_q X^q_\beta (K)\bigr)^{\negthinspace \mspace{1mu} hG/K} = 
\Bigl(\mspace{-2mu}(\holim_\beta X_\beta)^{hK}\Bigr)^{\negthinspace \mspace{-.5mu} hG/K}.\]

The second assertion now follows from the above equivalence: by 
applying (\ref{hfpss}), there is a homotopy fixed point spectral sequence 
$$E_2^{s,t}=H_c^s\bigl(G/K; \pi_t\bigl(\holim_\beta \holim_q X^q_\beta (K)\bigr)\bigr)\Rightarrow \pi_{t-s}\bigl((\holim_\beta X_\beta)^{hG}\bigr).$$ There is a $G/K$-equivariant isomorphism 
\[\pi_t\bigl((\holim_\beta X_\beta)^{hK}\bigr) \cong \pi_t\bigl(\holim_\beta \holim_q X^q_\beta (K)\bigr)\] 
of abelian groups, and thus, 
$\pi_t\bigl((\holim_\beta X_\beta)^{hK}\bigr)$ can be identified with the profinite $G/K$-module 
$\pi_t\bigl(\holim_\beta \holim_q X^q_\beta (K)\bigr)$.
\end{proof}

\begin{remark}
The discussion in Remark \ref{zyzzyva} applies to the fact that 
in Theorem \ref{iteratedthm}, we showed that $\holim_\beta \holim_q 
X^q_\beta(K)$ is a model for the $G/K$-spectrum 
$(\holim_\beta X_\beta)^{hK}$ in the 
category of profinite $G/K$-spectra. Thus, we just briefly note that 
if $Y$ is another profinite $G/K$-model for $(\holim_\beta X_\beta)^{hK}$, 
it is not known in general that there is an equivalence 
$(\holim_\beta \holim_q X^q_\beta(K))^{hG/K} \simeq Y^{hG/K}$. 
\end{remark}

The following definition extends Definition \ref{extend} 
to homotopy limits of diagrams.

\begin{defn}\label{beta}
Let $G$ be a strongly complete profinite group. Also, 
let $\{X_\beta\}_{\beta \in J}$, with $J$ a small 
category, be a diagram of $G$-spectra, such that for each $\beta$, $X_\beta$ is 
a $\pi$-finite $G$-spectrum that is fibrant as 
a spectrum. Then we define
\[(\holim_\beta X_\beta)^{hH} := (\holim_\beta F^s_G(X_\beta))^{hH},\] where 
$H$ is any closed subgroup of $G$. 
\end{defn}
\begin{remark}
We point out that the construction in Definition \ref{beta} has the 
following desired properties: (a) the canonical 
map \[\holim_\beta X_\beta \overset{\simeq}{\longrightarrow} 
\holim_\beta F^s_G(X_\beta)\] is a weak equivalence of spectra that is $G$-equivariant, so 
that $\holim_\beta X_\beta$ has the fibrant profinite $G$-spectrum $\holim_\beta F^s_G(X_\beta)$ as a model in the category of profinite $G$-spectra; 
(b) if the closed subgroup $H$ is finite, there are equivalences
\[(\holim_\beta X_\beta)^{h_dH} \overset{\simeq}{\longrightarrow} 
(\holim_\beta F^s_G(X_\beta))^{h_dH} \simeq (\holim_\beta F^s_G(X_\beta))^{hH} 
= (\holim_\beta X_\beta)^{hH},\] yielding agreement 
between the ``classical" and ``profinite" homotopy fixed points (the first and last terms, 
respectively, above); 
and (c) if $H$ is any closed subgroup that is strongly complete, then
\begin{align*}
(\holim_\beta F^s_G(X_\beta))^{hH} & \simeq \holim_\beta (F^s_G(X_\beta))^{hH} 
\simeq \holim_\beta (F^s_H(X_\beta))^{hH} \\ & \simeq (\holim_\beta F^s_H(X_\beta))^{hH},
\end{align*}
so that the two possible definitions of $(\holim_\beta X_\beta)^{hH}$ (the first 
and last terms in the preceding chain of equivalences) are equivalent 
to each other. 
%and (d) if $\{Y_\beta\}_{\beta \in J}$ is a diagram of $G$-spectra, 
%such that each $Y_\beta$ is $\pi$-finite and a fibrant spectrum, and 
%$\{h_\beta\}_\beta \colon \{X_\beta\}_\beta \to \{Y_\beta\}_\beta$ is a morphism between the two diagrams (that is, $\{h_\beta\}_\beta$ is a natural transformation 
%between the two functors, (for example, every map $h_\beta \colon X_\beta \to Y_\beta$ is $G$-equivariant))    
%(that is,
\end{remark}

Given Definition \ref{beta}, the following result is immediate from 
Theorem \ref{iteratedthm}. 

\begin{cor}\label{iteratedpifinitecor}
Let $G$ be a strongly complete profinite group and let K be a closed normal subgroup of $G$ 
with $G/K$ having finite virtual cohomological dimension. Also, let 
$\{X_\beta\}_{\beta \in J}$, where $J$ is a small category, be a diagram of $G$-spectra, with 
each $X_\beta$ both a fibrant spectrum and a $\pi$-finite $G$-spectrum. If the $J$-shaped 
diagram $\{F^s_G(X_\beta)\}_\beta$ of profinite $G$-spectra is a diagram of 
$K$-Postnikov $G$-spectra that is natural in $\beta$, then there is an equivalence of spectra 
\[
\bigl(\holim_\beta X_\beta\bigr)^{\negthinspace \mspace{1mu}hG} \simeq \Bigl(\mspace{-2mu}\bigl(\holim_\beta X_\beta\bigr)^{\negthinspace \mspace{1mu}hK}\Bigr)^{\scriptstyle{\negthinspace \mspace{-.7mu}hG/K}}
\]
and a conditionally convergent spectral sequence  
\[
E_2^{s,t}=H_c^s\bigl(G/K;\pi_t\bigl((\holim_\beta X_\beta)^{hK}\bigr)\bigr)\Rightarrow \pi_{t-s}\bigl((\holim_\beta X_\beta)^{hG}\bigr),
\] 
whose $E_2$-term is the continuous cohomology of $G/K$ with coefficients the profinite $G/K$-module 
$\pi_t\bigl((\holim_\beta X_\beta)^{hK}\bigr)$. 
\end{cor}
   
By Lemma \ref{groupresult}, the conditions in the first sentence of 
Corollary \ref{iteratedpifinitecor} are satisfied when the profinite group $G$ 
is $p$-adic analytic.   
   
\section{Iterated continuous homotopy fixed points for $E_n$}\label{LT}
We return to our main example of the extended Morava stabilizer group $G_n$ and 
its continuous action on the Lubin-Tate spectrum $E_n$. 
Let $BP$ be the Brown-Peterson spectrum for the fixed prime $p$. Its coefficient ring is $BP_*=\Z_{(p)}[v_1,v_2,\ldots]$, where $v_i$ has degree $2(p^i-1)$. There is a map
\[
r:BP_* \to E_{n\ast}=W(\F_{p^n})[[u_1,\ldots,u_{n-1}]][u,u^{-1}],
\] 
defined by $r(v_i)=u_iu^{1-p^i}$ for $i<n$, $r(v_n)=u^{1-p^n}$ and $r(v_i)=0$ for $i>n$, that makes $E_{n\ast}$ a $BP_*$-module (for more information on this and the remaining details in this paragraph, see, for example, \cite[Introduction]{DHonaction} and \cite[\S 7]{homasa}). Let $I$ be an ideal in $BP_*$ of the form $(p^{i_0},v_1^{i_1},\ldots, v_{n-1}^{i_{n-1}})$ that is $G_n$-invariant (that is, $G_n(IE_{n\ast}) = IE_{n\ast}$). Such ideals form a cofiltered system and 
their images in $E_{n\ast}$ under $r$ provide each $\pi_tE_n$ with the structure of a continuous profinite $G_n$-module: explicitly, we have
\[
\pi_tE_n \cong \lim_I \pi_tE_n/I\pi_tE_n.
\]
In fact, for $t$ odd these groups vanish, for $t$ even each quotient $\pi_tE_n/I\pi_tE_n$ is a finite discrete $G_n$-module (the $G_n$-action on each quotient is induced by the $G_n$-action on $\pi_tE_n$), and the above decomposition of $\pi_tE_n$ as a 
limit of these finite discrete $G_n$-modules is compatible with the $G_n$-action.

In the collection of ideals $I$, we can fix a descending chain of ideals \[I_0 \supset I_1 \supset \cdots 
\supset I_k \supset \cdots\] with an associated tower 
\[M_{I_0} \leftarrow M_{I_1} \leftarrow \cdots  \leftarrow M_{I_k} \leftarrow \cdots\] 
of generalized Moore spectra with trivial $G_n$-action, such that 
\[L_{K(n)}(\mathbb{S}^0) \simeq \holim_k \bigl(L_{E_n}(M_{I_k})\bigr)_{\negthinspace f}\,,\] where $L_{K(n)}$ denotes 
Bousfield localization with respect to $K(n)$ and $(-)_f$ is a fibrant replacement functor for Bousfield-Friedlander spectra, and for each $k \geq 0$, $BP_*(M_{I_k})=BP_*/I_k$ (see \cite{homasa} and \cite[\S 4]{HS}). 
It is useful to note 
that the last condition (about Brown-Peterson homology) implies that for each $k$, 
\[\pi_t(E_n \wedge M_{I_k})\cong \pi_tE_n/I_k\pi_tE_n\] for all $t$. Also, as in \cite{homasa}, there is an equivalence \[E_n \simeq \holim_k (E_n \wedge M_{I_k})_f\] which is induced by the isomorphism $E_n \cong E_n \wedge \Sph^0$ in the stable homotopy category and the $G_n$-equivariant map 
\[
E_n \wedge \Sph^0 \overset{\simeq}{\longrightarrow} \holim_k (E_n \wedge M_{I_k})_f
\] (in the source of this map, $G_n$ acts trivially on $\Sph^0$), 
whose underlying map of spectra is a stable equivalence. 

Now the profinite group $G_n$ is strongly complete (see e.g. \cite[page 330]{davis}) and each 
$E_n\wedge M_{I_k}$ is a $\pi$-finite $G_n$-spectrum. Thus, the functorial replacement of Theorem \ref{Gstablefinitecompletion} yields a fibrant profinite $G_n$-spectrum $F^s_{G_n}((E_n\wedge M_{I_k})_f)$ built out of pointed simplicial finite discrete $G_n$-sets and a $G_n$-equivariant map  
\[
(E_n\wedge M_{I_k})_f \overset{\simeq}{\longrightarrow} 
F^s_{G_n}((E_n\wedge M_{I_k})_f) =: \Enik
\]  
which is a stable equivalence of underlying spectra. It follows that there is an 
equivalence
\[\holim_k (E_n \wedge M_{I_k})_f \overset{\simeq}{\longrightarrow} 
\holim_k \Enik\] of spectra, and the target of this equivalence 
can be regarded as a homotopy limit in the category of profinite $G_n$-spectra. Thus, 
because of the equivalence \[E_n \simeq \holim_k \Enik,\] 
we denote by $E'_n$ the fibrant profinite $G_n$-spectrum
\[
E'_n := \holim_k \Enik.
\] 
In \cite{hfplt} the continuous homotopy fixed points $E_n^{hG}$ under a closed subgroup $G$ of $G_n$ are defined as \[E_n^{hG} := (E'_n)^{hG}.\] These homotopy fixed points satisfy 
\begin{equation}\label{usefullimit}
E_n^{hG} \simeq \holim_k (\Enik)^{hG},
\end{equation} 
since each $\Enik$ is a fibrant profinite $G$-spectrum.

Given the above setup, we are now ready to give a proof of 
Theorem \ref{iteratedthmEnintro}. This proof is an application 
of Theorem \ref{iteratedthm} and involves the use of Corollary 
\ref{goodtool} and Remark \ref{symondsremark} in a key way.

\begin{proof}[Proof of Theorem \ref{iteratedthmEnintro}] 
Recall that $G$ is an arbitrary closed subgroup of $G_n$ and 
$K$ is any closed normal subgroup of $G$. Since $G_n$ is a 
$p$-adic analytic group, Lemma \ref{groupresult} 
implies that $G$ and $G/K$ are strongly complete and 
$G/K$ has finite virtual cohomological 
dimension.

Since $\Enik$ is an $f$-$G_n$-spectrum, it is an $f$-$G$-spectrum, and it follows from 
\cite[Proposition 4.2.2]{symondsweigel} (as recalled in 
Remark \ref{symondsremark}) that $H^s_c(K; \pi_t(\Enik))$ is finite for all $s$ and $t$. 
Then, by applying Corollary \ref{goodtool}, we can conclude that $\Enik$ is a $K$-Postnikov 
$G$-spectrum for each $k$. For each $k$ and any integer $q$, let 
\[L^q_k = F^s_{G/K}\bigl(\colim_j \mathrm{Map}(EG, P^q 
\Enik)^{K_j}\bigr)\negthinspace \mspace{-2mu}:\] the 
diagram $\{L^q_k\}_{q \in \mathbb{Z}}$ is 
the inverse system of profinite $G/K$-models associated to the $K$-Postnikov $G$-spectrum 
$\Enik$. (To avoid any confusion, 
we remind the reader that $\{U_j\}_j$ is the collection of open normal 
subgroups of $G$ (not $G_n$) and $K_j = KU_j$.) It is easy to see 
that $\{\Enik\}_k$ is a diagram of $K$-Postnikov $G$-spectra that is natural in $k$, and hence, 
Theorem \ref{iteratedthm} yields the equivalence
\[E_n^{hG} = \bigl(\holim_k \Enik\bigr)^{hG} \simeq 
\Bigl(\mspace{-3mu}\bigl(\mspace{1mu}\holim_k \Enik\bigr)^{hK}\Bigr)^{\negthinspace hG/K} 
= \bigl(E_n^{hK}\bigr)^{\negthinspace hG/K}.
\qedhere
\] 
\end{proof} 

We now give the proof of Theorem \ref{thm:SS}. Our proof continues the 
setup and notation that was established above 
in the proof of Theorem \ref{iteratedthmEnintro}. 
    
\begin{proof}[Proof of Theorem \ref{thm:SS}.] 
Since $\{\Enik\}_k$ is a diagram of $K$-Postnikov $G$-spectra 
that is natural in $k$, Theorem 
\ref{iteratedthm} shows that there exists a conditionally convergent spectral 
sequence that has the form
\begin{equation}\label{firstdss}
H^s_c(G/K; \pi_t(E_n^{hK})) \Rightarrow \pi_{t-s}(E_n^{hG}).\end{equation}

Now let $H$ denote any closed subgroup of $G_n$ and notice that as at the 
end of the first paragraph of 
the proof of Theorem \ref{iteratedthmEnintro}, 
$H$ has finite virtual cohomological dimension. Thus, by 
applying Theorem 
\ref{profdisc} to the $f$-$H$-spectrum $\Enik$, we obtain for each $k$ the 
equivalence $(\Enik)^{hH} \simeq (\Enik)^{h_dH}$. We can build upon this 
equivalence to obtain, for each $k$, the equivalences
\begin{equation}\label{chainH}
(\Enik)^{hH} \simeq (\Enik)^{h_dH} \simeq \bigl(\bigl(\colim_{N \vartriangleleft_o G_n} 
E_n^{dhN}\bigr) 
\wedge M_{{I_k}}\bigr)^{h_dH} 
\simeq E_n^{dhH} \wedge M_{{I_k}},
\end{equation} 
where the above colimit is over the open 
normal subgroups of $G_n$, the second equivalence follows from the fact that the composition 
\[\bigl(\colim_{N \vartriangleleft_o G_n} E_n^{dhN}\bigr) \wedge M_{{I_k}} 
\overset{\simeq}{\longrightarrow} E_n \wedge M_{{I_k}} 
\overset{\simeq}{\longrightarrow} (E_n \wedge M_{{I_k}})_f \overset{\simeq}{\longrightarrow} 
\Enik\] is a weak equivalence in the category of discrete $H$-spectra (the key 
ingredient here is that the first map above is a weak equivalence of spectra: this is due to 
\cite{devinatzhopkins} and is made explicit in \cite[Theorem 6.3, Corollary 6.5]{davis}), and 
the last equivalence in (\ref{chainH}) 
is by \cite[Corollary 9.8]{davis} and \cite[Theorem 8.2.1]{behrensdavis}. By \cite[proof of Lemma 3.5]{devinatz2}, 
the spectrum $E_n^{dhH} \wedge M_{{I_k}}$ has finite homotopy groups, and hence, 
so does the spectrum $(\Enik)^{hH}$.

Our last conclusion implies that 
$\pi_\ast\bigl((\Enik)^{hK}\bigr)$ and $\pi_\ast\bigl((\Enik)^{hG}\bigr)$ 
are degreewise finite. Therefore, it follows from 
(\ref{usefullimit}) (and from a second application of (\ref{usefullimit}) 
with $G$ set equal to $K$), as in \cite[proof of Proposition 7.4]{homasa} 
(see also the beginning 
of the proof of Theorem 7.6 in \cite{davis2}), 
that the spectral sequence in (\ref{firstdss}) is the 
inverse limit over $\{k\}$ of conditionally convergent 
spectral sequences $E_r(k) \equiv E_r^{\ast, \ast}(k)$ that have the 
form
\[E_2^{s,t}(k) = H^s_c\bigl(G/K; \pi_t\bigl((\Enik)^{hK}\bigr)\bigr) 
\Rightarrow \pi_{t-s}\bigl((\Enik)^{hG}\bigr),\]
where for each $k$, $E_r(k)$ is 
constructed as an instance of the descent spectral sequence in (\ref{hfpss}), with abutment 
equal to $\pi_\ast\bigl((\holim_q L^q_k)^{hG/K}\bigr)$.  

We pause to introduce some helpful terminology: if $Z^\bullet$ is a cosimplicial spectrum that is fibrant in each codegree, then we refer 
to the conditionally convergent homotopy spectral sequence 
\[E_2^{s,t} = H^s\bigl[\pi_t(Z^\ast)\bigr] \Rightarrow \pi_{t-s}\bigl(\holim_\Delta 
Z^\bullet\bigr)\] as the 
homotopy spectral sequence for $\holim_\Delta 
Z^\bullet$. For clarity later, we go ahead and point out that 
as defined above, whenever we use this 
terminology, the relevant homotopy limit (that is, $\holim_\Delta Z^\bullet$) 
is always indexed by $\Delta$. 

By \cite[Proposition 3.20 and Lemma 3.21,$\mspace{-3mu}$~(b)]{hfplt}, 
spectral sequence $E_r(k)$ can be regarded as the homotopy spectral sequence for 
\[\holim_\Delta \mathrm{Map}_{G/K}\bigl((G/K)^{\bullet+1}, \holim_q L^q_k\bigr).\] Thus, 
\cite[Lemma 3.5]{hfplt} implies that $E_r(k)$ is the homotopy spectral 
sequence for \[\holim_\Delta \holim_q \mathrm{Map}_{G/K}\bigl((G/K)^{\bullet+1}, L^q_k\bigr).\] 
Since $L^q_k$ is an $f$-$G/K$-spectrum, the isomorphism in (\ref{cosimp}) 
(in the proof of Theorem \ref{profdisc}) shows 
that $E_r(k)$ is isomorphic to the homotopy spectral sequence for 
\[\holim_\Delta \holim_q \mathrm{Map}_c\bigl((G/K)^{\bullet+1}, L^q_k\bigr)^{G/K}.\] It follows
that $E_r(k)$ is isomorphic to spectral sequence 
$\mathcal{H}_r(k)$, which is defined to be the homotopy spectral sequence  
for 
\[\holim_\Delta \holim_{q \geq 0} \mathrm{Map}_c\bigl((G/K)^{\bullet+1}, L^q_k\bigr)^{G/K} 
\simeq \Bigl(\holim_{q \geq 0} L^q_k\Bigr)^{\negthinspace \mspace{0mu}
hG/K}\]
($\mathcal{H}_r(k)$ is an instance of the spectral sequence that is studied in 
\cite[Theorem 8.8]{davis}), where the above equivalence uses 
that $L^q_k$ is a fibrant spectrum for each $q \geq 0$ 
and $G/K$ has finite virtual cohomological dimension, 
and the expression on the right-hand side 
above is the $G/K$-homotopy fixed points of the continuous 
$G/K$-spectrum $\holim_{q \geq 0} L^q_k$ (in the sense of \cite{davis}).

We want to compare 
spectral sequence $\mathcal{H}_r(k)$ with a certain other spectral sequence, and to 
accomplish this, we need to do some preliminary work. The first step is to note that 
for each $q \geq 0$, there is a canonical map $\Enik \to P^q\Enik$ in the category of profinite 
$G_n$-spectra, and hence, 
there is the induced $G/K$-equivariant map
\begin{equation}\label{maptouse}
\colim_j \mathrm{Map}(EG, \Enik)^{K_j} \to \colim_j \mathrm{Map}(EG, P^q \Enik)^{K_j}.
\end{equation} The source of 
the map in (\ref{maptouse}) is the first term in the following chain of 
equivalences:
\begin{align*}
\colim_j \mathrm{Map}(EG, \Enik)^{K_j} & = \colim_j (\Enik)^{hK_j} 
\simeq \colim_j (E_n^{dhK_j} \wedge M_{{I_k}}) 
\\ & \simeq 
E_n^{dhK} \wedge M_{{I_k}},\end{align*} where the second equivalence 
follows from (\ref{chainH}) and the last equivalence is due to 
\cite[Proposition 6.3]{devinatzhopkins} (for the details, see \cite[Lemma 7.1]{davis2}).
As recalled earlier, the spectrum 
$E_n^{dhK} \wedge M_{{I_k}}$ has finite homotopy groups, and hence, 
the source $\colim_j \mathrm{Map}(EG, \Enik)^{K_j}$ of the map in (\ref{maptouse}) is a 
$\pi$-finite $G/K$-spectrum.

For each $q \geq 0$, there is a morphism
\[\colim_j \mathrm{Map}(EG, \Enik)^{K_j} \to L^q_k\] of discrete $G/K$-spectra 
that is equal to the composition
\[\colim_j \mathrm{Map}(EG, \Enik)^{K_j} \to 
F^s_{G/K}\bigl(\colim_j \mathrm{Map}(EG, \Enik)^{K_j}\bigr) 
\to L^q_k,\] where the first map in the above composition exists by 
Theorem \ref{Gstablefinitecompletion}, since 
the source is a $\pi$-finite $G/K$-spectrum, and the 
last map is obtained by applying $F^s_{G/K}$ to the 
map in (\ref{maptouse}). It follows that for each $q \geq 0$, there is an induced map
\[
\mathrm{Map}_c\bigl((G/K)^{\bullet+1}, \colim_j \mathrm{Map}(EG, 
\Enik)^{K_j}\bigr)^{\negthinspace G/K} \to \mathrm{Map}_c\bigl((G/K)^{\bullet+1}, 
L^q_k\bigr)^{G/K}
\] of cosimplicial spectra, and hence, there is a map
\[
\mathrm{Map}_c\bigl((G/K)^{\bullet+1}, \colim_j \mathrm{Map}(EG, 
\Enik)^{K_j}\bigr)^{\negthinspace G/K} \to \lim_{q \geq 0} \mathrm{Map}_c\bigl((G/K)^{\bullet+1}, L^q_k\bigr)^{G/K}
\] of cosimplicial spectra. 
Composition of the last map above with the canonical map
from the limit to the homotopy limit yields a map
\[
\mathrm{Map}_c\bigl((G/K)^{\bullet+1}, \colim_j \mathrm{Map}(EG, 
\Enik)^{K_j}\bigr)^{\negthinspace G/K} \to \holim_{q \geq 0} \mathrm{Map}_c\bigl((G/K)^{\bullet+1}, L^q_k\bigr)^{G/K}
\] that induces a morphism from $\mathcal{D}_r(k)$, which is defined to be 
the homotopy spectral sequence  for 
\[
\holim_\Delta \mathrm{Map}_c\bigl((G/K)^{\bullet+1}, \colim_j \mathrm{Map}(EG, 
\Enik)^{K_j}\bigr)^{\negthinspace G/K},\] to spectral sequence 
$\mathcal{H}_r(k)$.

The isomorphism between spectral sequences $\mathcal{H}_r(k)$ and $E_r(k)$, 
established earlier, implies that $\mathcal{H}_r(k)$ has $E_2$-term isomorphic to 
\[E_2^{s,t}(k) = H^s_c\bigl(G/K; \pi_t\bigl((\Enik)^{hK}\bigr)\bigr).\] Also, as in 
\cite[Theorem 7.9]{davis}, spectral sequence $\mathcal{D}_r(k)$ has 
$E_2$-term equal to 
\[H^s_c\bigl(G/K; \pi_t\bigl(\colim_j \mathrm{Map}(EG, 
\Enik)^{K_j}\bigr)\bigr),\] with $\pi_t\bigl(\colim_j \mathrm{Map}(EG, \Enik)^{K_j}\bigr)$ 
equal to a discrete $G/K$-module. Since 
\[\colim_j \mathrm{Map}(EG, 
\Enik)^{K_j} \simeq E_n^{dhK} \wedge M_{{I_k}} \simeq (\Enik)^{hK},\] where the 
last equivalence above applies (\ref{chainH}), and the 
finite profinite $G/K$-module $\pi_t\bigl((\Enik)^{hK}\bigr)$ is automatically 
a discrete $G/K$-module, the 
$E_2$-terms of $\mathcal{D}_r(k)$ and $\mathcal{H}_r(k)$ are isomorphic. Therefore, 
spectral sequences $E_r(k)$ and $\mathcal{D}_r(k)$ are 
isomorphic to each other from the $E_2$-terms onward. 

In \cite[proof of Theorem 7.6]{davis2}, there is a descent 
spectral sequence that is referred to as $E_r^{\ast, \ast}(K, G, k)$ and which 
has the form
\begin{equation}\label{lastdss}
H^s_c(G/K; \pi_t(E_n^{h'K} \wedge M_{I_k})) \Rightarrow 
\pi_{t-s}(E_n^{h'G} \wedge M_{I_k}).
\end{equation} 
As done with the notation ``$(E_n \wedge M_{I_k})_f$", if $Z$ denotes an arbitrary spectrum, 
we let $Z_f$ be a functorial fibrant replacement of 
$Z$ in the stable model category of spectra. It follows from 
\cite[(6.2)]{davis2} that spectral sequence $E_r^{\ast, \ast}(K, G, k)$ is 
the homotopy spectral sequence for 
\[\holim_\Delta \mathrm{Map}_c\bigl((G/K)^{\bullet+1}, \colim_{N \vartriangleleft_o G_n} 
(E_n^{dhNK} \wedge M_{I_k})_f\bigr)^{G/K}.\] Because 
\[\colim_j \mathrm{Map}(EG, \Enik)^{K_j} \simeq E_n^{dhK} \wedge M_{I_k} 
\simeq \colim_{N \vartriangleleft_o G_n}(E_n^{dhNK} \wedge M_{I_k})_f,\] where the 
last equivalence follows from \cite[Definition 1.5]{devinatzhopkins} (for some 
details, see \cite[Lemma 5.2]{davis2}), there is an isomorphism between 
spectral sequences $\mathcal{D}_r(k)$ and $E_r^{\ast, \ast}(K, G, k)$ (to see 
this, it is helpful to recall the 
definition of $\mathcal{D}_r(k)$). 

Since the spectral sequence 
in (\ref{firstdss}) is the spectral sequence $\lim_k E_r(k)$, it is isomorphic 
to the spectral sequence $\lim_k \mathcal{D}_r(k)$, and hence, 
to the spectral sequence $\lim_k E_r^{\ast, \ast}(K,G,k)$, from the $E_2$-term onward. 
Since $\lim_k E_r^{\ast, \ast}(K,G,k)$ is the descent spectral sequence
\[H^s_c(G/K; \pi_t(E_n^{h'K})) \Rightarrow \pi_{t-s}(E_n^{h'G})\] of (\ref{davis2SS}) (see 
\cite[proof of Theorem 7.6]{davis2}), which is 
isomorphic to the strongly convergent spectral sequence
\[H^s_c(G/K; \pi_t(E_n^{dhK})) \Rightarrow \pi_{t-s}(E_n^{dhG})\] 
of (\ref{lhsSS}) from the $E_2$-term onward, it follows that the spectral 
sequence in (\ref{firstdss}) is isomorphic to spectral sequences (\ref{lhsSS}) and 
(\ref{davis2SS}) from the $E_2$-term onward, and its strong convergence 
follows from that of (\ref{lhsSS}).  
\end{proof}
     
\bibliographystyle{amsplain}

\end{document}